\documentclass[a4paper,11pt]{amsart}
\usepackage{geometry}
\usepackage{amsmath}
\usepackage{amsthm}
\usepackage{amsfonts}
\usepackage{amssymb}
\usepackage{mathrsfs}
\usepackage{mathtools}
\usepackage{graphicx}
\usepackage{todonotes} 
\usepackage[all]{xy}
\usepackage{tikz}
\usepackage{pgf}
\usepackage{pgffor}
\usetikzlibrary{arrows,positioning,matrix,decorations.pathmorphing,calc,fadings,decorations.pathreplacing} 
\usepackage{keyval}
\usepgfmodule{shapes}
\usepgfmodule{plot}
\usetikzlibrary{decorations}
\numberwithin{equation}{section}
\usepackage[hidelinks]{hyperref}
\usepackage{hhline}
\usepackage{multirow}
\usepackage{multicol}

\newtheorem{thm}{Theorem}
 \numberwithin{thm}{section}
\newtheorem{prop}[thm]{Proposition}

\newtheorem{defi}[thm]{Definition}
\newtheorem{ej}[thm]{Example}
\newtheorem{example}[thm]{Example}
\newtheorem{lemma}[thm]{Lemma}
\newtheorem{remark}[thm]{Remark}
\newcommand{\R}{\mathbb{R}}

\newcommand{\Q}{\mathbb{Q}}

\newcommand{\Z}{\mathbb{Z}}
\newcommand{\mA}{\mathcal{A}}

\newcommand{\intal}{\mathrm{Int}}

\DeclareMathOperator{\minor}{minor}
\newcommand{\Sp}{\mathscr{S}}

\newcommand{\uri}{\rightarrow_\circ}
\newcommand{\arrowschem}[2]{\raisebox{-2ex}%
	{$\stackrel{\stackrel{\displaystyle#1}{\longrightarrow}}%
	{\stackrel{\longleftarrow}{#2}}$}}

\begin{document}

\title[Lower bounds and regions of multistationarity]{Lower bounds 
for positive roots and regions of multistationarity in chemical reaction networks}

\author{Fr\'ed\'eric Bihan}
\address{Laboratoire de Math\'ematiques\\
         Universit\'e Savoie Mont Blanc\\
         73376 Le Bourget-du-Lac Cedex\\
         France}
\email{Frederic.Bihan@univ-savoie.fr}
\urladdr{http://www.lama.univ-savoie.fr/~bihan/}

\author{Alicia Dickenstein}
\address{Dto.\ de Matem\'atica, FCEN, Universidad de Buenos Aires, and IMAS (UBA-CONICET), Ciudad Universitaria, Pab.\ I, 
C1428EGA Buenos Aires, Argentina}
\email{alidick@dm.uba.ar}
\urladdr{http://mate.dm.uba.ar/~alidick}

\author{Magal\'{\i} Giaroli}
\address{Dto.\ de Matem\'atica, FCEN, Universidad de Buenos Aires, and IMAS (UBA-CONICET), Ciudad Universitaria, Pab.\ I, 
C1428EGA Buenos Aires, Argentina}
\email{mgiaroli@dm.uba.ar}


\date{}

\begin{abstract}
Given a real sparse polynomial system, we present a general framework to find explicit coefficients 
for which the system has more than one positive solution, based on the recent 
article by Bihan, Santos and Spaenlehauer \cite{bihan}. We apply this approach
to find explicit {\it reaction rate constants} and {\it total conservation constants} 
in biochemical reaction networks for which the associated dynamical system is multistationary.
\end{abstract}
\maketitle

\section{Introduction}

Multistationarity is a key property of biochemical reaction networks, 
because it provides a mechanism for switching between different response states. 
This enables multiple outcomes for cellular-decision making in cell signaling systems.  
Questions about steady states in biochemical reaction networks under \emph{mass-action kinetics} 
are fundamentally questions about nonnegative real solutions to parametrized polynomial ideals. 
In particular, multistationarity corresponds to the existence of more than one positive steady state 
with fixed conserved quantities.  We refer the reader to~\cite{alicia} for an expansion 
of this point of view and further references.

 In this work,  we develop tools from real algebraic geometry based on the paper~\cite{bihan} by Bihan, Santos and Spaenlehauer,
 to analyze systems biology models.  We present a general framework to describe \emph{multistationarity regions} in parameter 
space, that is, to find ``explicit'' parameters for which 
 multistationarity occurs. We exemplify our theoretical results in different biochemical networks of interest of arbitrary
 size and number of variables.  For this, we need to adapt the theoretical results to make them amenable to effective computations
in a variety of specific instances in the modeling of biochemical systems. Our developments are also based on
 the existence of explicit parametrizations of the corresponding steady state varieties, as described in
 Theorems~28 and 35 in~\cite{aliciaMer}.
 
 We give two complementary approaches. On one side, we show how to deform a given 
 choice of reaction rate constants and total concentration constants in order to produce multistationarity. 
 On the other side, we describe open multistationarity regions in the space of all these constants.
 We derive inequalities in the reaction constants  and in the total conservation constants
 whose validity implies the presence of multistationarity. In
 particular,  our results are the key tools in our companion paper~\cite{cascades} to identify multistationarity regions for 
enzymatic cascades 
 of Goldbeter--Koshland loops~\cite{GK81} of arbitrary number of layers,
when a same phosphatase catalyzes the transfer of phosphate groups at two different layers.

\subsection{Basics on chemical reaction networks and multistationarity} 

Given  a  set  of
$s$ chemical  species, a \textit{chemical reaction network} 
on this set of species is a finite directed graph whose vertices are labeled by complexes 
and whose edges $\mathcal{R}$ represent the reactions and are labeled by parameters $\kappa\in \R^{|\mathcal{R}|}_{>0}$, 
which are called \textit{reaction rate constants}. 
Complexes determine vectors in $\Z^s_{\geq 0}$ according to the stoichiometry of the species they consist of.
We identify a complex with its corresponding vector and also with the formal linear combination of
species specified by its coordinates.
 Under mass-action kinetics, the chemical reaction network defines the following autonomous system of ordinary differential equations
in the concentrations $x_1, x_2, \dots, x_s$ of the species as \emph{functions of time $t$}:
 
 \begin{equation}\label{fam} \dot{x} = \left(\frac{dx_1}{dt},\frac{dx_2}{dt},\dots,\frac{dx_s}{dt}\right) = f(x) =
\sum_{y\rightarrow{y'}\in\mathcal{R}}\kappa_{yy'} x^{y}(y'-y),
 \end{equation}
 where $x=(x_1,x_2,\dots, x_s)$, $x^y = x_1^{y_1}x_2^{y_2}\dots x_s^{y_s}$ and $y\rightarrow{y'}$ indicates that 
$(y,y')\in\mathcal{R}$; that is, the complex $y$ reacts to the complex $y'$. 
The right-hand side of each differential equation $\frac{dx_i}{dt}$ is a polynomial $f_i$ in $x_1,x_2,\dots,x_s$ with real 
 coefficients.
 A concentration vector $\bar x\in \R^s_{\geq 0}$ is a \emph{steady state} of the system if $f(\bar x)=0$, and $\bar x$ 
 is a \emph{positive} steady state if moreover $\bar x\in \R^s_{>0}$.
 We observe that the vector $\dot{x}(t)$ lies for all time $t$ 
in the so called \emph{stoichiometric subspace}  $S$  which is the linear subspace spanned by 
the reaction vectors $\{y'-y :  y\rightarrow{y'}\in \mathcal{R}\}$. Thus, a trajectory $x(t)$ 
beginning at a positive vector $x(0)=x^0\in\R^s_{> 0}$ remains in the \emph{stoichiometric 
 compatibility class} $(x^0+S)\cap\R^s_{\geqslant 0}$ for all $t\geq 0$. The linear equations of $x^0+S$ are 
 called the \emph{conservation laws}. Given a linear function $\ell$ vanishing on $S$ and any fixed 
 stoichiometric compatibility class, $\ell$ takes a constant value over all points in this
 class. We will refer to these constant values as {\it total conservation constants}.
  
 We say that the network \emph{has the capacity for multistationarity} 
 if there exists a choice of reaction rate constants $\kappa$ such that there are two or more steady states in one 
 stoichiometric compatibility class for some initial state $x^0$, 
that is, for an appropriate choice of total conservation constants.   
 Starting with~\cite{cfI, cfII}, several articles studied 
 the capacity for multistationarity from the structure of the digraph
 \cite{banajipantea,feliu, fw13, fc, anne,mueller, aliciaMer}.  
 Once the capacity for multistationarity is determined, the following difficult 
 question is to find values of  multistationary parameters as exhaustively and explicitly as possible. 
 This problem is in principle effectively computable but the inherent high complexity does not allow 
 to treat interesting networks with standard general tools. Several articles addressed this task, 
 providing different answers based on ad-hoc computations,  injectivity results based on signs of minors, 
 and degree theory \cite{feliumincheva,cfr,mincheva,gatermann,kfc,kothamachu, sontag}.

\subsection{Our results for a two-component system}
We showcase our results in a simple meaningful example.
The following chemical reaction network is a \textit{two-component 
system}~\cite{twocomponents} with \textit{hybrid} histidine kinase (hybrid $HK$) whose multistationarity 
 was studied in~\cite{feliumincheva,kothamachu}.  Two-component signal transduction systems enable bacteria to sense, respond, and 
adapt to a wide 
 range of environments, stressors, and growth conditions.  
 This network has six species $X_1, \dots,X_6$,  ten complexes (e.g. $X_1$ or $X_1+X_6$, 
 also identified with the vectors $e_1$ and $e_1+e_6$ in $\Z_{\ge 0}^ 6$) and  six reactions (directed edges),
with labels given by positive reaction rate constants $k_1, \dots, k_6$:
  \begin{align}\label{networkHHK}
  X_1 \xrightarrow{k_1} X_2 &\xrightarrow{k_2} X_3  \xrightarrow{k_3} X_4 \nonumber\\
  X_3 + X_5 &\xrightarrow{k_4} X_1 + X_6  \\
  X_4 + X_5 &\xrightarrow{k_5} X_2 + X_6 \nonumber \\
  X_6 &\xrightarrow{k_6} X_5 \nonumber
  \end{align}
    This labeled digraph represents the following biological mechanism.
 Two component signaling relies on 
 phosphotransfer reactions between histidine and aspartate residues on histidine kinases ($HKs$) and response regulator ($RR$) proteins.  
 The hybrid $HK$ consists of two phosphorylable domains. We denote the phosphorylation 
 state of each site by $p$ if the site is phosphorylated and $0$ if it is not; 
 the four possible states of $HK$ are denoted by $HK_{00}$, $HK_{p0}$, $HK_{0p}$, and $HK_{pp}$. 
 We let $RR$ be the unphosphorylated response regulator protein, and $RR_p$ the phosphorylated form. 
 Upon receiving a signal, the $HK$ can auto-phosphorylate. 
 Whenever the second phosphorylation site is occupied, the phosphate group can be transferred to $RR$.   
 In \eqref{networkHHK}, we displayed the corresponding network of reactions
denoting by $X_1, \dots, X_6$ the chemical species $HK_{00}, HK_{p0}, HK_{0p}, HK_{pp}, RR, RR_p$, respectively. 
 
  In what follows, we denote the concentration of the chemical species $X_1,\dots, X_6$  by lower-case letters $x_1,\dots,x_6$. 
These  concentrations are assumed to be functions which evolve in time $t$, 
according to the following polynomial autonomous dynamical system:
   \begin{align*}
   \frac{dx_1}{dt} &= f_1(x) = {-k_1}x_1 + k_4x_3x_5, & \frac{dx_2}{dt} &= f_2(x) = k_1x_1 - k_2x_2 + k_5x_4x_5, \\  
   \frac{dx_3}{dt} &= f_3(x) = k_2x_2 -k_3x_3 -k_4x_3x_5, & \frac{dx_4}{dt} &= f_4(x)= k_3x_3 -k_5x_4x_5,\\
   \frac{dx_5}{dt} &= f_5(x) = -k_4x_3x_5 - k_5x_4x_5 + k_6x_6, & \frac{dx_6}{dt} &= f_6(x) = k_4x_3x_5 + k_5x_4x_5  - k_6x_6.
   \end{align*}  
 It is straightforward to check that there are two linearly independent relations: $f_1+f_2+f_3+f_4=f_5+f_6=0$, 
which imply the existence of two constants $T_1, T_2$ such that for any value of 
  $t$:
  \begin{align}\label{eq:conservationlawsnetworkHHK}
  \ell_1(x)= x_1 + x_2 + x_3 + x_4 = &\  T_1, \\
  \ell_2(x) = x_5 + x_6 = &\  T_2.\nonumber
  \end{align}
   We assume that the linear variety cut out by these equations intersects the positive orthant, so $T_1, T_2$ are also positive 
   parameters. These parameters $T_1$, $T_2$ are the total conservation constants and 
the linear equations $\ell_1$ and $\ell_2$ are the conservation laws.

 We now explain our strategy in the previous network (\ref{networkHHK}). 
Our problem is to determine values of 
$(k_1, \dots,k_6,T_1,T_2)$ in $\R_{>0}^8$ for which the polynomial system
 \[ f_1(x) = \dots = f_6(x)= \ell_1(x) - T_1= \ell_2(x)-T_2=0,\]
 has {\it more than one positive solution}  $x \in \R^6_{>0}$. 
 We have, using the framework of the main Theorems~\ref{th:BS} and~\ref{thm:main2}:
 
 \begin{thm}\label{th:examplehhk}
 With the notation of (\ref{networkHHK}) and (\ref{eq:conservationlawsnetworkHHK}), 
 assume that a fixed choice of reaction rate constants 
 satisfies the condition $k_3 > k_1$. Then,  $ k_6\left(\dfrac{1}{k_2}+\dfrac{1}{k_3}\right) < k_6\left(\dfrac{1}{k_1}+\dfrac{1}{k_2}\right)$ and for any
 choice of total concentration constants veriying the inequalities
 \begin{equation} \label{eq:conditionhhk}
 \begin{aligned}
 k_6\left(\dfrac{1}{k_2}+\dfrac{1}{k_3}\right) < & \, \, \frac{T_1}{T_2} \, < k_6\left(\dfrac{1}{k_1}+\dfrac{1}{k_2}\right),
 \end{aligned}
 \end{equation}
 there exist positive constants $N_1, N_2$ such that for any values of $\beta_4$ and $\beta_5$ satisfying $\beta_4 > N_1$ and 
  $\frac{\beta_5}{\beta_4} > N_2$, the system has 
 at least three positive steady states after modifying only the 
 parameters $k_4, k_5$  via the rescaling $\overline{k_4} = \beta_4 \, k_4,\, \overline{k_5} = \beta_5 \, k_5$.

 \end{thm}
  
  In \cite{feliumincheva}, the authors present necessary 
  and sufficient conditions for the multistationarity of the network. They prove that the region of the reaction rate constant space 
for which multistationarity exists is completely characterized by the inequality $k_3 > k_1$, but they do not describe 
  the particular stoichiometric compatibility classes for which there are multistationarity.  Instead, our   inequalities~\eqref{eq:conditionhhk} 
  give conditions on the total concentrations and the reaction rate constants for the ocurrence of  multistationarity. In 
\cite{kothamachu} necessary and 
  sufficient conditions on all the parameters of the system for bistability are provided,  with a treatment ad hoc using Sturm's Theorem.

  \subsection{The contents of the paper}
  The basic idea we develop in this paper is to detect in the convex hull of the support of the monomials that define 
 the equations of the steady states, (at least two) simplices \emph{positively decorated} (see Definition~\ref{def:dec}) 
 that form part of a regular subdivision. This ensures the extension of the positive 
 real solutions of the corresponding subsystems to the total system.  In Sections~\ref{sec:1} and~\ref{sec:2} we state and explain the theoretical 
 setting which is of general interest for the search of positive solutions of sparse 
 real polynomial systems beyond the applications we consider. In Section~\ref{sec:1} we work with the same support 
 for all the polynomials of the system. In Section~\ref{sec:2} we present a mixed approach to the results 
 of Section~\ref{sec:1}, considering different supports for each polynomial.
  We refer the reader to~\cite{triangulations,GKZ} for the definitions and main properties 
of the combinatorial objects we deal with.  

Our main results in these sections are Theorems~\ref{thm:main2} and~\ref{thm:mixedgamma}.
In the following sections, we apply these results for a class of biochemical reaction networks under mass-action kinetics.
This application is not straightforward and requires known and new results on the structure of their steady states.

 In Section~\ref{sec:3}, we study the sequential distributive multisite phosphorylation systems with any number $n$ of phosphorylation sites. 
Such systems were studied by many authors, starting with Wang and Sontag~\cite{sontag}, who gave bounds and conditions 
  for  monostationarity and multistationarity in the parameters, with an interesting treatment ad hoc which also allowed them 
  to find improved lower bounds (see also~\cite{kfc}). In \cite{mincheva}, Conradi and Mincheva showed using degree theory 
  and computations with the aid of a computer algebra system, that catalytic constants determine the capacity 
  for multistationarity in the dual phosphorylation mechanism. They also indicate in the case $n=2$ how to find values of the total 
  concentrations such that multistationarity occurs.  The more general interesting approach in \cite{feliumincheva} is also based  
  on degree theory. The authors show how to find conditions on the reaction rates to guarantee mono or multistationarity, 
  but they do not describe the particular total concentration constants for which there are multiple 
  equilibria. 
  With our approach, we obtain for any $n$ a system of three polynomial equations in three variables that describes the steady states, 
 in the framework of \cite{bihan}. 
  We give conditions on all the parameters (both on the reaction constants and on the total concentration constants) so that there are at least two positively 
  decorated simplices in a regular subdivision of the convex hull of our support and by rescaling the rest of the parameters, 
  we guarantee the existence of at least two non-degenerate positive steady states (see Theorem~\ref{th:nfosfo}).
 
Then, in Section~\ref{sec:4}, we show that these systems
  as well as the two-component system in the Introduction, are a particular example of a
class of biological systems introduced in~\cite{aliciaMer}  called MESSI systems, which contains many other important mechanisms.
We focus on a particular class of MESSI systems called $s$-toric MESSI systems, which includes the sequential phosphorylation systems, 
for which explicit monomial parametrizations of the steady states are given in~\cite{aliciaMer}. We prove general results for 
$s$-toric 
MESSI systems, that in particular explain our computations in Section~\ref{sec:3}. Theorem~\ref{th:reescalamientos} is the key 
to apply the framework of Theorem~\ref{thm:main2} to describe multistationarity regions for all these biological systems.

 \section{Positive solutions of sparse polynomial systems}\label{sec:1}
 
Along this section, we fix a finite point configuration
\[\mathcal{A} = \{a_1, \dots, a_n\}\subset \Z^d, \, n\geq d+2, \]
 and we  assume that the convex hull of $\mathcal{A}$ is a full-dimensional polytope.
A subset of $\mathcal A$ consisting of affinely
 independent points will be called a simplex; we will also say that it is a $d$-simplex when the dimension of
its convex hull is $d$.

\subsection{Regular subdivisions}
A \textit{regular subdivision of $\mA$} is induced by a height function $h:\mA \to \R$, $h=(h(a_1), \dots, h(a_n))$ as follows.  
Consider the \textit{lower convex hull} of the lifted configuration $\mA^h= \{(a_1, h(a_1)), \dots, (a_n,h(a_n))\}
\subset \R^{d+1}$, which is
the union of the faces of the convex hull of $\mA^h$ for which the inner normal directions have positive last coordinate. 
The associated regular subdivision $\Gamma_h$
is the union of the subsets $\mA_F=\{a_i \, : \, (a_i, h(a_i)) \in F\}$ of $\mA$ which are the projections back to $\mA$ of 
lifted points in a face $F$
of this lower convex hull.

It is useful to have a more geometric picture of this
subdivision, like the one depicted in Figure~\ref{fig:Regular triangulation}, but it is important to note that 
these subsets cannot be identified in general with their convex hulls, which are convex polytopes with integer vertices, but with 
their
\textit{marked} convex hulls containing all the points $a_j \in \mA$ for which the affine linear function which 
interpolates the values of $h$ at the vertices takes  the value $h(a_j)$ at $a_j$ (so, other points besides the vertices can 
occur).
A regular subdivision is called a \textit{regular triangulation} of $\mA$ if the only points in each subset of 
the subdivision are the vertices of their convex hull and these vertices are affinely independent.  
Figure~\ref{fig:NonRegularTriangulation} depicts a triangulation into simplices which is not regular, that is, which cannot be 
induced by any height function $h$.

 The set of all height vectors inducing a regular subdivision $\Gamma$
of $\mathcal{A}$ is defined by a finite number of linear inequalities. 
 Thus, this set is a finitely generated convex cone $\mathcal{C}_\Gamma$ in $\R^n$ with apex at the origin.  
 When $\Gamma$ is a triangulation, the cone $\mathcal{C}_\Gamma$ 
 is full dimensional (cut out by strict inequalities).
 All these facts and many more are described in \cite[Ch.7]{GKZ}.
 
  \begin{center}
 \begin{figure}
\begin{minipage}[t][4cm]{.45\textwidth}
\centering
 \begin{tikzpicture}%
 	[x={(1.595303cm, 0.021152cm)},
 	y={(0.905733cm, 0.517961cm)},
 	z={(-0.024093cm, 0.8770689cm)},
 	scale=0.900000,
 	back/.style={dotted,very thick},
 	edge/.style={color=black, thick},
 	facet/.style={fill=blue!95!black,fill opacity=0.800000},
 	vertex/.style={inner sep=1.2pt,circle,draw=black,fill=black,thick,anchor=base}]
 %
 %
 \coordinate (0.00000, 0.00000, 0.00000) at (0.00000, 0.00000, 0.00000);
 \coordinate (3.00000, 0.00000, 0.00000) at (3.00000, 0.00000, 0.00000);
 \coordinate (3.00000, 2.00000, 0.00000) at (3.00000, 2.00000, 0.00000);
 \coordinate (0.00000, 2.00000, 0.00000) at (0.00000, 2.00000, 0.00000);
 \coordinate (0.00000, 0.00000, 2.50000) at (0.00000, 0.00000, 2.50000);
 \coordinate (3.00000, 0.00000, 2.50000) at (3.00000, 0.00000, 2.50000);
 \coordinate (3.00000, 2.00000, 2.50000) at (3.00000, 2.00000, 2.50000);
 \coordinate (0.00000, 2.00000, 2.50000) at (0.00000, 2.00000, 2.50000);
 \coordinate (2.00000, 1.00000, 0.00000) at (2.00000, 1.00000, 0.00000);
 \coordinate (1.00000, 0.70000, 0.00000) at (1.00000, 0.70000, 0.00000);
 \coordinate (1.00000, 1.30000, 0.00000) at (1.00000, 1.30000, 0.00000);
 \coordinate (2.00000, 1.00000, 1.00000) at (2.00000, 1.00000, 1.00000);
 \coordinate (1.00000, 0.70000, 1.30000) at (1.00000, 0.70000, 1.30000);
 \coordinate (1.00000, 1.30000, 2.30000) at (1.00000, 1.30000, 2.30000);
 
 

 \draw[edge,back] (0.00000, 0.00000, 0.00000) -- (0.00000, 0.00000, 2.50000);
 \draw[edge,back](0.00000, 2.00000, 0.00000) -- (0.00000, 2.00000, 2.50000);
 \draw[edge,back](3.00000, 2.00000, 0.00000) -- (3.00000, 2.00000, 2.50000);
 \fill[facet, fill=yellow!,fill opacity=1.000000] (0.00000, 2.00000, 2.50000) -- (1.00000, 0.70000, 1.30000) -- (2.00000, 1.00000, 1.00000) -- cycle {};
 \fill[facet, fill=violet!50!white,fill opacity=1.000000] (0.00000, 2.00000, 2.50000) -- (3.00000, 2.00000, 2.50000) -- (2.00000, 1.00000, 1.00000) -- cycle {};
 \fill[facet, fill=orange!,fill opacity=1.000000] (0.00000, 2.00000, 2.50000) -- (1.00000, 0.70000, 1.30000) -- (0.00000, 0.00000, 2.50000) -- cycle {};
 \draw[edge] (0.00000, 2.00000, 2.50000) -- (2.00000, 1.00000, 1.00000);
 \fill[facet, fill=yellow!50!white,fill opacity=1.000000] (0.00000, 2.00000, 0.00000) -- (1.00000, 0.70000, 0.00000) -- (2.00000, 1.00000, 0.00000) -- cycle {};
 \fill[facet, fill=violet!20!white,fill opacity=1.000000] (0.00000, 2.00000, 0.00000) -- (3.00000, 2.00000, 0.00000) -- (2.00000, 1.00000, 0.00000) -- cycle {};
 \fill[facet, fill=orange!50!white,fill opacity=1.000000] (0.00000, 2.00000, 0.00000) -- (1.00000, 0.70000, 0.00000) -- (0.00000, 0.00000, 0.00000) -- cycle {};
 \draw[edge,back](1.00000, 1.30000, 0.00000) -- (1.00000, 1.30000, 2.30000);
 \draw[edge] (1.00000, 0.70000, 1.30000) -- (0.00000, 2.00000, 2.50000);
 \fill[facet, fill=blue!,fill opacity=1.000000] (2.00000, 1.00000, 1.00000) -- (3.00000, 2.00000, 2.50000) -- (3.00000, 0.00000, 2.50000) -- cycle {};
 \fill[facet, fill=red!,fill opacity=1.000000] (0.00000, 0.00000, 2.50000) -- (1.00000, 0.70000, 1.30000) -- (3.00000, 0.00000, 2.50000) -- cycle {};
 \fill[facet, fill=green!,fill opacity=1.000000] (3.00000, 0.00000, 2.50000) -- (1.00000, 0.70000, 1.30000) -- (2.00000, 1.00000, 1.00000) -- cycle {};

 \fill[facet, fill=blue!50!white,fill opacity=1.000000] (2.00000, 1.00000, 0.00000) -- (3.00000, 2.00000, 0.00000) -- (3.00000, 0.00000, 0.00000) -- cycle {};
 \fill[facet, fill=red!50!white,fill opacity=1.000000] (0.00000, 0.00000, 0.00000) -- (1.00000, 0.70000, 0.00000) -- (3.00000, 0.00000, 0.00000) -- cycle {};
 \fill[facet, fill=green!50!white,fill opacity=1.000000] (3.00000, 0.00000, 0.00000) -- (1.00000, 0.70000, 0.00000) -- (2.00000, 1.00000, 0.00000) -- cycle {};
 \draw[edge] (0.00000, 0.00000, 2.50000) -- (3.00000, 0.00000, 2.50000);
 \draw[edge] (0.00000, 0.00000, 2.50000) -- (1.00000, 0.70000, 1.30000);
 \draw[edge] (3.00000, 0.00000, 2.50000) -- (1.00000, 0.70000, 1.30000);
 \draw[edge] (2.00000, 1.00000, 1.00000) -- (1.00000, 0.70000, 1.30000);
 \draw[edge] (3.00000, 0.00000, 2.50000) -- (2.00000, 1.00000, 1.00000);
 \draw[edge] (0.00000, 0.00000, 0.00000) -- (3.00000, 0.00000, 0.00000);
 \draw[edge] (0.00000, 0.00000, 0.00000) -- (0.00000, 2.00000, 0.00000);
 \draw[edge] (3.00000, 0.00000, 0.00000) -- (3.00000, 2.00000, 0.00000);
 \draw[edge] (0.00000, 2.00000, 0.00000) -- (3.00000, 2.00000, 0.00000);
 \draw[edge] (3.00000, 0.00000, 0.00000) -- (2.00000, 1.00000, 0.00000);
 \draw[edge] (3.00000, 2.00000, 0.00000) -- (2.00000, 1.00000, 0.00000);
 \draw[edge] (0.00000, 2.00000, 0.00000) -- (2.00000, 1.00000, 0.00000);
 \draw[edge] (1.00000, 0.70000, 0.00000) -- (2.00000, 1.00000, 0.00000);
 \draw[edge] (1.00000, 0.70000, 0.00000) -- (3.00000, 0.00000, 0.00000);
 \draw[edge] (1.00000, 0.70000, 0.00000) -- (0.00000, 0.00000, 0.00000);
 \draw[edge] (1.00000, 0.70000, 0.00000) -- (0.00000, 2.00000, 0.00000);
 \draw[edge,back] (3.00000, 0.00000, 2.50000) -- (3.00000, 0.00000, 0.00000);
 \draw[edge,back](2.00000, 1.00000, 0.00000) -- (2.00000, 1.00000, 1.00000);
 \draw[edge,back] (1.00000, 0.70000, 0.00000) -- (1.00000, 0.70000, 1.30000);
 \draw[edge] (3.00000, 0.00000, 2.50000) -- (3.00000, 2.00000, 2.50000);
 \draw[edge] (3.00000, 2.00000, 2.50000) -- (2.00000, 1.00000, 1.00000);
 \draw[edge] (0.00000, 0.00000, 2.50000) -- (0.00000, 2.00000, 2.50000);
 \draw[edge] (3.00000, 2.00000, 2.50000) -- (0.00000, 2.00000, 2.50000);
 \node[vertex] at (0.00000, 0.00000, 0.00000)     {};
 \node at (-0.2,0,0) {$a$};
 \node[vertex] at (0.00000, 2.00000, 0.00000)     {};
 \node[vertex] at (3.00000, 0.00000, 0.00000)     {};
 \node[vertex] at (3.00000, 2.00000, 0.00000)     {};
 \node at (4.05,0.6,0) {$\R^d$};
 \node at (4.2,0.6,2) {$\R^{d+1}$};
 \node[vertex] at (0.00000, 0.00000, 2.50000)     {};
 \node[vertex] at (0.00000, 2.00000, 2.50000)     {};
 \node[vertex] at (3.00000, 0.00000, 2.50000)     {};
  \node at (-0.6,0,2.5) {$(a,h(a))$};
 \node[vertex] at (3.00000, 2.00000, 2.50000)     {};
 \node[vertex] at (2.00000, 1.00000, 0.00000)   	 {};
 \node[vertex] at (1.00000, 1.30000, 0.00000) 	 {};
 \node[vertex] at (1.00000, 0.70000, 0.00000)	 {};
 \node[vertex] at (2.00000, 1.00000, 1.00000)   	 {};
 \node[vertex] at (1.00000, 0.70000, 1.30000) 	 {};
 \node[vertex] at (1.00000, 1.30000, 2.30000)	 {};

 \end{tikzpicture}
 \caption{Regular triangulation.} \label{fig:Regular triangulation}
 \vspace{\baselineskip}
 \end{minipage}\hfill
 \begin{minipage}[t][4cm]{.45\textwidth}
 \centering
  \begin{tikzpicture}%
    	[scale=0.900000,
    	back/.style={loosely dotted, thin},
    	edge/.style={color=black, thick},
    	facet/.style={fill=red!70!white,fill opacity=0.800000},
    	vertex/.style={inner sep=1pt,circle,draw=black,fill=black,thick,anchor=base
    }]
    %
    %
    \coordinate (0.00000, 0.00000) at (0.00000, 0.00000);
    \coordinate (4.00000, 0.00000) at (4.00000, 0.00000);
    \coordinate (0.00000, 4.00000) at (0.00000, 4.00000);
    \coordinate (1.00000, 1.00000) at (1.00000, 1.00000);
    \coordinate (1.00000, 2.00000) at (1.00000, 2.00000);
    \coordinate (1.00000, 2.00000) at (2.00000, 1.00000);
    \draw[edge] (0.00000, 0.00000) -- (4.00000, 0.00000);
    \draw[edge] (0.00000, 0.00000) -- (0.00000, 4.00000);
    \draw[edge] (0.00000, 4.00000) -- (4.00000, 0.00000);
    \draw[edge] (0.00000, 0.00000) -- (1.00000, 1.00000);
    \draw[edge] (0.00000, 0.00000) -- (1.00000, 2.00000);
    \draw[edge] (4.00000, 0.00000) -- (1.00000, 1.00000);
    \draw[edge] (4.00000, 0.00000) -- (2.00000, 1.00000);
    \draw[edge] (0.00000, 4.00000) -- (1.00000, 2.00000);
    \draw[edge] (0.00000, 4.00000) -- (2.00000, 1.00000);
    \draw[edge] (1.00000, 1.00000) -- (2.00000, 1.00000);
    \draw[edge] (1.00000, 2.00000) -- (2.00000, 1.00000);
    \draw[edge] (1.00000, 1.00000) -- (1.00000, 2.00000);
    \node[vertex] at (0.00000, 0.00000) {};
    \node[vertex] at (4.00000, 0.00000) {};
    \node[vertex] at (0.00000, 4.00000) {};
    \node[vertex] at (1.00000, 1.00000) {};
    \node[vertex] at (1.00000, 2.00000){};
    \node[vertex] at (2.00000, 1.00000){};
    \end{tikzpicture}
    \caption{Non-regular triangulation.} \label{fig:NonRegularTriangulation}
    \end{minipage}
    \vspace{\baselineskip}
 \end{figure}
 \end{center}
 
We will denote by 
$A \in \Z^{(d+1) \times n}$ the integer matrix:
\begin{equation}\label{eq:A}
A=\left(
\begin{array}{ccc}
1 &  \ldots &1 \\
a_1& \ldots   & a_{n} 
\end{array}
\right),
\end{equation}
and by $A_h \in \R^{(d+2) \times n}$ the matrix:
\begin{equation}\label{eq:Ah}
A_h=\left(
\begin{array}{ccc}
1 &  \ldots &1 \\
a_1& \ldots   & a_{n} \\
h_1 &\ldots  & h_n
\end{array}
\right).
\end{equation}
Note that our assumption that the convex hull of $\mA$ has dimension $d$ is equivalent to assuming that the rank of $A$ is equal to $d+1$.

 Let $\Delta=\{a_{i_1},\dots,a_{i_{d+1}}\}$ be a $d$-simplex with vertices in $\mathcal{A}$.
Let $I=\{i_1,i_2, \ldots,i_{d+1}\}$ and assume $i_1 <i_2< \cdots < i_{d+1}$. 
  Denote by $d_{I}$ the determinant of the $(d+1) \times (d+1)$ submatrix of $A$ with columns indicated by $I$, 
  which is nonzero because we are assuming
  that $\Delta$ is a simplex.  Also, for any index $i \notin I$, denote by $d_{I \cup \{i\}}(h)$
   the determinant of the $(d+2) \times (d+2)$ submatrix of $A_h$
  with columns indicated by $I \cup \{i\}$, multiplied by the sign of the permutation
  that sends the set of indices in $I \cup \{i\}$ ordered by $<$ to $(i_1, \dots, i_{d+1}, i)$ with $i$ as the last index.
Using the Laplace expansion of the determinant  along the last row, we see that
$d_{I \cup \{i\}}(h)$ is an affine linear function of $h$ which equals $\langle m^I_i, h \rangle$, 
where $m^I_i$ is a vector in the kernel of $A$ with support included $I \cup \{i\}$ and with nonzero $i$-th coordinate, for any $i \notin I$.

  Consider the cone $\mathcal{C}_{\Delta}$ of all height vectors inducing 
 a regular subdivision of $\mA$ that contains $\Delta$. 
Observe that $\mathcal{C}_{\Delta}$ is non-empty; for instance, any vector  $h \in \R^{n}$ with $h_i=0$ for any $i \in I$ and 
$h_i  >0$ for any index
  $i \notin I$, belongs to $\mathcal{C}_{\Delta}$. Moreover, $C_\Delta$ is an open rational polyhedral cone, described as follows:

  \begin{lemma}\label{lem:CDelta} With the previous notations, we have:
 \[\mathcal{C}_{\Delta}=\{(h_1,\dots,h_n)\in \R^n \, : \,  \langle d_I \cdot m^I_i, h \rangle \, = \, d_I \cdot d_{I \cup \{i\}}(h)  > 0 \},\]
and the $n-(d+1)$ vectors $d_I \cdot m^I_i, i \notin I,$ are a basis of the kernel of $A$.
 \end{lemma}
The proof of Lemma~\ref{lem:CDelta} is straightforward.
Let $p\ge 1$ and consider  $\Delta_1, \dots, \Delta_p$ $d$-simplices 
  in $\mathcal{A}$. We denote by ${\mathcal C}_{\Delta_1,\dots, \Delta_p}$ the  cone of all height vectors $h$ 
  defining a regular subdivision of $\mA$ that contains $\Delta_1, \dots,\Delta_p$. 
 We deduce from Lemma~\ref{lem:CDelta} the following description.

\begin{lemma}\label{lem:CDeltap}
Let  $\Delta_1, \dots, \Delta_p$ be simplices 
  in $\mathcal{A}$ which occur in a regular subdivision of $A$. If the index set of  the vertices of 
 $\Delta_k$ is  $I_k$ for any $k=1, \dots, p$,   then the nonempty open polyhedral cone
  ${\mathcal C}_{\Delta_1, \dots,\Delta_p}$ is defined by the linear inequalities
 \begin{equation} \label{E:conefortwo}
\langle d_{I_k} \cdot m^{I_k}_i , h \rangle  > 0 \quad \text{ for all }  k = 1, \dots, p,  \text{ and all } i \notin I_k,
 \end{equation}
and the vectors $\{  d_{I_k} \cdot m^{I_k}_i , k=1,\dots, p, i \notin I_k\}$ generate the kernel of $A$.
\end{lemma}

  \begin{figure}[]
	
  \centering
 \begin{tikzpicture}
 	[scale=1.200000,
 	back/.style={loosely dotted, thin},
 	edge/.style={color=black, thick},
 	facet/.style={fill=red!70!white,fill opacity=0.800000},
 	vertex/.style={inner sep=1.5pt,circle,draw=black,fill=black,thick,anchor=base
 }]
 %
 %
 \coordinate (0.00000, 0.50000) at (0.00000, 0.50000);
 \coordinate (2.00000, 1.00000) at (2.00000, 1.00000);
 \coordinate (1.00000, 0.00000) at (1.00000, 0.00000);
 \coordinate (1.00000, 2.00000) at (1.00000, 2.00000);
\fill[facet, fill=red!40!white,fill opacity=0.800000] (1.00000, 2.00000) --  (1.00000, 0.00000)-- (0.00000,0.50000) -- cycle {};
\fill[facet, fill=blue!80!white,fill opacity=0.800000] (1.00000, 2.00000) --  (1.00000, 0.00000)-- (2.00000,1.00000) -- cycle {};

 \draw[edge] (0.00000, 0.50000) -- (1.00000, 2.00000);
 \draw[edge] (0.00000, 0.50000) -- (1.00000, 0.00000);
 \draw[edge] (1.00000, 0.00000) -- (1.00000, 2.00000);
 \draw[edge] (1.00000, 0.00000) -- (2.00000, 1.00000);
 \draw[edge] (1.00000, 2.00000) -- (2.00000, 1.00000);
 \node[vertex] at (0.00000, 0.50000){};
 \node[vertex] at (1.00000, 0.00000){};
 \node[vertex] at (1.00000, 2.00000){};
\node[vertex] at (2.00000, 1.00000){};
\node at (2.00000, 0.50000){$\Delta_2$};
 \node at (0.00000, 1.20000){$\Delta_1$};
 \end{tikzpicture}
  \hspace{2cm}
 \begin{tikzpicture}
 	[scale=1.200000,
 	back/.style={loosely dotted, thin},
 	edge/.style={color=black, thick},
 	facet/.style={fill=red!70!white,fill opacity=0.800000},
 	vertex/.style={inner sep=1.5pt,circle,draw=black,fill=black,thick,anchor=base
 }]
 %
 %
 \coordinate (0.00000, 0.00000) at (0.00000, 0.00000);
 \coordinate (1.00000, 1.00000) at (1.00000, 1.00000);
 \coordinate (2.30000, 1.00000) at (2.30000, 1.00000);
 \coordinate (0.80000, 2.00000) at (0.80000, 2.00000);
\fill[facet, fill=red!40!white,fill opacity=0.800000] (0.80000, 2.00000) --  (1.00000, 1.00000)-- (0.00000,0.00000) -- cycle {};
\fill[facet, fill=blue!80!white,fill opacity=0.800000] (1.00000, 1.00000) --  (0.00000, 0.00000)-- (2.30000,1.00000) -- cycle {};

 \draw[edge] (0.00000, 0.00000) -- (1.00000, 1.00000);
 \draw[edge] (0.00000, 0.00000) -- (0.80000, 2.00000);
 \draw[edge] (0.00000, 0.00000) -- (2.30000, 1.00000);
 \draw[edge] (0.80000, 2.00000) -- (1.00000, 1.00000);
 \draw[edge] (1.00000, 1.00000) -- (2.30000, 1.00000);
  \draw[edge] (1.00000, 1.00000) -- (2.30000, 1.00000);
  \draw [dashed] (0.80000, 2.00000) -- (2.30000, 1.00000);
 \node[vertex] at (0.00000, 0.00000){};
 \node[vertex] at (1.00000, 1.00000){};
 \node[vertex] at (0.80000, 2.00000){};
\node[vertex] at (2.30000, 1.00000){};
\node at (2.00000, 0.50000){$\Delta_2$};
 \node at (0.00000, 1.20000){$\Delta_1$};
 \end{tikzpicture}
 \caption{Examples of simplices $\Delta_1$ and $\Delta_2$, which share a facet, $d=2$.} \label{fig:2simplicescommonfacet}
\end{figure}
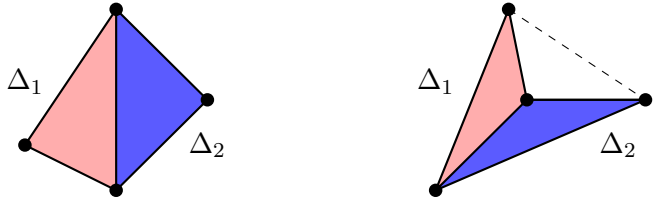

 \begin{defi}
 \label{def:share}
  We will say that two simplices $\Delta_1, \Delta_2 \subset \mA$
 \textit{share a facet}  if the intersection of their convex hulls is a facet of both, that is, a face of codimension one.  
 See Figure \ref{fig:2simplicescommonfacet}.  
 \end{defi}
  We will need the following remark:

\begin{remark}\label{rem:circuit}  
{\rm A point configuration  $\mathcal B  =\{b_1, \dots, b_{d+2}\}$ with $d+2$ points which span $\R^{d}$ and such that 
any proper subset is affinely independent, is called a \textit{circuit}. Any circuit $\mathcal B$ has exactly two triangulations 
$\Gamma_{\pm}$, which are furthermore regular.
 They can be described in the following way
 (see \cite[chapter 7, Proposition 1.2.]{GKZ}).
 Consider any nonzero vector $\lambda \in \R^{d+2}$ such that $\sum_{i=1}^{d+2} \lambda_i=0$
 and $\sum_{i=1}^{d+2} \lambda_i b_i=0$ (in other words, any nontrivial affine relation on ${\mathcal B}$). 
 Note that all coordinates of $\lambda$ are nonzero.
 Write $[d+2]=\{1, \dots, d+2\}$ as the disjoint union $N_+ \sqcup N_-$, with $N_{+} = \{ i \in [d+2]\, : \, \lambda_i >0\}$ and
 $N_- = \{ i \in [d+2]\, : \, \lambda_i < 0\}$. The $d$-simplices of $\Gamma_{+}$ are the sets $[d+2] 
 \setminus \{ i\}$ for $ i \in N_{+}$.
 Similarly, the $d$-simplices of $\Gamma_{-}$ are the sets $[d+2] \setminus \{ i\}$ for $ i \in N_{-}$.}
 \end{remark}
 
 We are ready to prove the following proposition, that we will need in our applications.
 
\begin{prop}\label{prop:2}
Let  $\Delta_1$, $\Delta_2$ be two  $d$-simplices 
  in $\mathcal{A}$ which share a facet.
 Then, there exists a regular subdivision of $\mathcal{A}$ containing $\Delta_1$ and $\Delta_2$, 
 and so the cone $\mathcal C_{\Delta_1, \Delta_2}$ is nonempty.
 Moreover, there exists a regular triangulation containg both
 simplices.
 \end{prop}

 \begin{proof}
 The configuration $\mathcal B=\Delta_1 \cup \Delta_2$ has cardinality $d+2$ and it is a circuit. 
 As we remarked $\mathcal B$ has exactly two regular triangulations $\Gamma_{\pm}$.
 Without loss of generality, assume ${\mathcal B}= \{ a_1, \dots, a_{d+2}\}$, with
 $F=\{a_1, \dots, a_d\}$ the common facet of $\Delta_1$ and  $\Delta_2$.  
 Let $\lambda \in \Z^{d+2}$ be a  nontrivial affine relation on ${\mathcal B}$. 
 As $a_{d+1}$ and $a_{d+2}$ lie in opposite sides of the hyperplane passing 
 through $F$,  it holds that $\lambda_{d+2}$ and $ \lambda_{d+1}$ have the same sign. Therefore $\Delta_1$ and $\Delta_2$ belong 
 to 
 the same regular triangulation, say $\Gamma_+$. 
 
 Let $h:\mathcal B \to \R$ be a height
 function inducing $\Gamma_+$. Let $\varphi_1, \dots, \varphi_\ell$ be the affine linear
 functions  which interpolate the values of $h$ at each of the $d$-simplices of
 $\Gamma_+$ and set $\varphi = {\rm max}\{\varphi_1, \dots, \varphi_\ell\}$.
 For any choice of generic positive values $h_{d+3}, \dots, h_n$ verifying
 $h_j > \varphi(a_j)$ for any $j=d+3, \dots, n$,  the height function $h': \mA\to \R$ 
 that extends $h$ by defining $h'(a_j)= h_j, j=d+3, \dots, n$, induces a
 regular triangulation of $\mA$ extending $\Gamma_+$, and so in particular, it contains $\Delta_1$ and $\Delta_2$.
 \end{proof}

 \begin{remark} {\rm Under the notations of Lemma \ref{lem:CDeltap} with $p=2$,
if $\Delta_1$ and $\Delta_2$ share a common facet with vertices $a_i$ with $i \in I$,  then the inequality
corresponding to $k=1$ and $i \in I_2 \setminus I$ coincides with the one corresponding to $k=2$ and
 $i \in I_1 \setminus I$, so that we might forget one of these inequalities in \eqref{E:conefortwo}
 to get $2(n-d-1)-1=2n-2d-3$ inequalities defining $\mathcal{C}_{\Delta_1,\Delta_2}$. 
 This generalizes the circuit case where any of the two regular triangulations is  determined by one of its simplices.}
 \end{remark}

 \subsection{Decorated simplices and lower bounds for the number of positive solutions}
 
 Consider a sparse polynomial system in $d$ variables $x=(x_1, \dots, x_d)$ with support included in
 $\mathcal A$ and coefficient matrix $C = (c_{ij}) \in \R^{d \times n}$:
 \begin{equation}\label{systemf}
 f_{1}(x)=\dots=f_{d}(x)=0,
 \end{equation}
 with
 \begin{equation*}
 f_{i}(x)=\sum_{j=1}^n c_{ij} \, x^{a_j} \in \R[x_1,\dots,x_d], \ i=1,\dots,d.
 \end{equation*}

 Let $\Gamma$ be a regular subdivision of $\mathcal A$ and $h \in \mathcal{C}_\Gamma$.
 Consider the following family of polynomial systems parametrized by a positive real
 number $t$:
 \begin{equation}\label{systemwitht}
 f_{1,t}(x)=\dots=f_{d,t}(x)=0,
 \end{equation}
 where 
 \begin{equation*}
 f_{i,t}(x)=\sum_{j=1}^n c_{ij}\, t^{h(a_j)} \, x^{a_j} \in \R[x_1,\dots,x_d], \ i=1,\dots,d, \ t>0.
 \end{equation*}
 For each positive real value of $t$, this system has again support included in $\mathcal{A}$. Recall that
 a common root of~\eqref{systemwitht} is nondegenerate when it is not a zero of the Jacobian of $f_{1,t}, \dots, f_{d,t}$.

 Following Section 3 in~\cite{bihan}, we define:

 \begin{defi} A $d\times (d+1)$ matrix $M$ with real entries is called positively spanning
  if all the values $(-1)^i\minor(M,i)$ are nonzero and have the same sign, 
 where $\minor(M,i)$ is the determinant of the square matrix obtained by removing the $i$-th column.
 \end{defi}
 Equivalently, a matrix is positively spanning if all the coordinates of any non-zero vector 
 in the kernel of the matrix are non-zero and have the same sign.
 \begin{defi} \label{def:dec}
 Let $C$ be a $d\times n$ matrix with real entries. We say that a $d$-simplex 
 $\Delta=\{a_{i_1},\dots,a_{i_{d+1}}\}$ in $\Gamma$ is positively decorated by $C$  
 if the $d\times(d+1)$ submatrix of $C$ with columns indicated by $\{i_1,\dots,i_{d+1}\}$ is positively spanning.
 \end{defi}

 The following result is a generalization of Theorem 3.4 in~\cite{bihan}, 
 with essentially the same proof, combined with Proposition~\ref{prop:2}:
 
 \begin{thm} \label{th:BS}
Let $\mathcal A$ and $p$ $d$-simplices $\Delta_1, \dots, \Delta_p$
which occur in a regular subdivision $\Gamma$ of  $\mathcal A$, and which are positively decorated 
by a matrix $C  \in \R^{d \times n}$. Let $h$ be a height function that defines $\Gamma$.
Then, there exists $t_0\in\R_{>0}$ such that for all $0<t<t_0$, the number of (nondegenerate) solutions of (\ref{systemwitht}) 
 contained in the positive orthant is at least $p$.
In particular, if there are two $d$-simplices with vertices in $\mathcal A$ sharing a facet which are both positively decorated by $C$, 
the number of positive solutions of \eqref{systemwitht} is at least two.
 \end{thm}

 
 The idea of the proof consists first in observing that the system obtained by
 considering only the monomials in
 a positively decorated $d$-simplex has exactly one nondegenerate positive solution. 
 Then, we can take a compact set $K$ in the positive orthant which contains all the nondegenerate
positive solutions of the restricted systems, for each positively decorated $d$-simplex. If $t > 0$ is small enough, 
we can obtain pairwise disjoint sets (of the form $t^{\alpha}\dot K, \alpha \in \R^d$), such that each one contains at 
least one non-degenerate positive of the system \eqref{systemwitht}. We note that in \cite{bihan}, 
the result is for a regular triangulation but it also holds if we have a regular subdivision.
 
 We will give a similar result in Theorem~\ref{thm:main2} below, but our focus is to describe a subset 
  with nonempty interior in the space of coefficients
 where we can bound from below the number of positive solutions of the associated system.
We start with a general result about convex polyhedral cones.

\begin{prop}\label{prop:cone}
Let $L$ be a linear subspace of $\R^n$ of dimension $\ell_1$ together with a basis $\{v_1, \dots, v_{\ell_1}\}$.  
Let $m_1, \dots, m_{\ell}$ be a system of generators of $L^\perp$  such that the open polyhedral cone
$${\mathcal C} =\{h \in \R^n \, : \, \langle m_r, h \rangle > 0,  \quad r= 1, \dots, \ell \}$$
is non-empty.
For any $\varepsilon \in \R_{>0}^\ell$, denote by $\mathcal C_\varepsilon$ the $n$-dimensional convex polyhedral cone
\begin{equation} \label{eq:cone}
{\mathcal C}_\varepsilon \, = \, \{ h \in \R^n \, : \, 
\langle m_r, h \rangle > \varepsilon_r,  \quad r= 1, \dots, \ell \}.
\end{equation}
Consider the map
$\varphi: \R_{>0}^{\ell_1} \times\R_{>0} \times \R^n \to \R^n_{>0}$:
\[ \varphi(\alpha,t,h) \, = \, (  t^{h_1} \prod_{j=1}^{\ell_1} \alpha_j^{v_{j1}}, \dots,   t^{h_n} \prod_{j=1}^{\ell_1} \alpha_j^{v_{jn}}).\]
Then, we have:
\begin{eqnarray}\label{eq:varphi}
\varphi(\R_{>0}^{\ell_1} \times(0,t_0) \times {\mathcal C}_\varepsilon) \, &=& \, \{ \gamma \in \R_{>0}^n \, : \, \gamma^{m_r} < t_0^{\varepsilon_r}, \, 
r=1 \dots,\ell\}\ \mbox{and}\\ \varphi(\R_{>0}^{\ell_1} \times (0,t_0] \times \bar{\mathcal C}_\varepsilon) \, &=& \, 
\{ \gamma \in \R_{>0}^n \, : \, \gamma^{m_r} \leq t_0^{\varepsilon_r}, \, r=1 \dots,\ell\},
\end{eqnarray}
where $\bar{\mathcal C}_\varepsilon$ denotes the closure of ${\mathcal C}_\varepsilon$.
\end{prop}

\begin{proof}
We first prove that a positive vector $\gamma$ is of the form $\gamma =\varphi(\alpha,t,h)$  if and only if
\[ \gamma^{m_r} = t^{\langle m_r, h \rangle},  \quad r=1, \dots, \ell.\]
The only if part is straightforward, taking into account that 
we are assuming that for any $r,j$ it holds that $\langle m_r, v_j\rangle =0$:
\[ \varphi(\alpha,t,h)^{m_r}\, = \, t^{\langle m_r, h \rangle} \,  \prod_{j=1}^{\ell_1} \alpha_j^{\langle m_r, v_j\rangle} 
 \, = \, t^{\langle m_r, h \rangle}.\]
On the other side, if $\gamma^{m_r} = t^{\langle m_r, h \rangle}$  for any $r=1, \dots, \ell$, then the vector 
$$\gamma_{t,h}=(\gamma_1 \, t^{-h_1}, \dots, \gamma_n \, t^{-h_n})$$
verifies that 
$ \gamma_{t,h}^m \, = \, 1, \,  \text{ for any }  m \in L^\perp.$
Thus, taking coordinatewise logarithms, we get that
\[ \langle m , {\rm log}( \gamma_{t,h}) \rangle \, = \, 0 \, \text{ for any }  m \in L^\perp,\]
which means that ${\rm log}( \gamma_{t,h}) \in L$.  Then, there exist real constants $\lambda_1, \dots, \lambda_\ell$ 
such that ${\rm log}( \gamma_{t,h}) = \sum_{j=1}^\ell \lambda_j \, v_j$. Calling $\alpha\in \R_{>0}^\ell$ 
the vector with coordinates $\alpha_j = e^{\lambda_j}$ we get that $\gamma= \varphi(\alpha,t,h)$, as wanted.

Now, assuming $0 < t < t_0 < 1$ and $\langle m_r, h \rangle > \varepsilon_r$ for all $r = 1 \dots, \ell$, 
we have that $t^{\langle m_r, h \rangle} < t_0^{\varepsilon_r}$ and moreover 
$(0, t_0^{\varepsilon_r}) = \{ t^{\langle m_r, h \rangle},  0 < t < t_0, h \in
{\mathcal C}_\varepsilon\}$, which proves both containments. The other equality follows immediately.
\end{proof} 
 
We now present the main result of this section.
 
 \begin{thm}\label{thm:main2}
 Consider a set $\mathcal{A}=\{a_1, \dots, a_{n}\}$ of $n$ points in $\R^d$ and 
 a matrix $C=(c_{i,j}) \in \R^{d \times n}$.  Assume  there are $p$ $d$-simplices $\Delta_1, \dots, \Delta_p$ contained in
 $\mathcal{A}$, which are part of a regular subdivision of $\mathcal A$
and are positively decorated by a $d \times n$ matrix $C$.

Let $m_1\dots,m_\ell \in \R^n$ be vectors that define a presentation of the cone   ${\mathcal C}_{\Delta_1,\dots, \Delta_p}$ 
of all height vectors $h \in \R^n$ that induce a regular subdivision of $\mA$ 
 containing $\Delta_1, \dots, \Delta_p$: 
 \begin{equation} \label{eq:cone2}
 {\mathcal C}_{\Delta_1,\dots, \Delta_p} \, = \, \{h \in \R^n \, : \, \langle m_r, h \rangle  > 0 , \; r=1,\ldots,\ell\}.
 \end{equation}
Then, for any $\varepsilon \in (0,1)^\ell$ there exists $t_0(\varepsilon) >0$ such that for any $\gamma$ in the open set
\[ U \, = \, \cup_{\varepsilon \in (0,1)^\ell} \,  \{ \gamma \in \R_{>0}^n \, ; \, \gamma^{m_r} < t_0(\varepsilon)^{\varepsilon_r}, \, r=1 \dots,\ell\},\]
 the system 
 \begin{equation}\label{systemgamma}
 \sum_{j=1}^n c_{ij}\, \gamma_{j} \, x^{a_j}=0 , \;  \ i=1,\dots,d,
 \end{equation}
 has at least $p$ nondegenerate solutions in the positive orthant. In particular, given two
  $d$-simplices that share a facet, the system~\eqref{systemgamma} has at least $2$
 nondegenerate positive solutions.
  \end{thm}

\begin{proof}
Let $L$ be the linear subspace generated by the rows of the matrix $A$, and let
$v_1, \dots, v_{d+1}$ denote its row vectors, which are a basis of $L$ because we are
assuming that $A$ has rank $d+1$.  With this choice, we  can apply Proposition~\ref{prop:cone} 
 to the cone ${\mathcal C} ={\mathcal C}_{\Delta_1,\dots, \Delta_p}$, by Lemma~\ref{lem:CDeltap}.  Note that the map 
$\varphi: \R_{>0}^{d+1} \times\R_{>0} \times \R^n \to \R^n_{>0}$ equals in this case:
\[ \varphi(\alpha,t,h) \, = \, ( \alpha^{(1 a_1)} \, t^{h_1}, \dots,  \alpha^{(1 a_n)}\, t^{h_n}).\]

We denote by $\overline{\mathcal C}_\varepsilon$ the closure of the cone $\mathcal C_\varepsilon$ defined in~\eqref{eq:cone}.
Let $B$ denote the unit ball in $\R^n$. In the proof of Theorem 3.4 in~\cite{bihan}, and thus in the proof of Theorem~\ref{th:BS}, 
one can see that given any $h$, it is possible to find a positive number $t_0$ for which the conclusion of Theorem~\ref{th:BS} holds 
for any $h'$ close to $h$. As $B \cap \overline{\mathcal C}_\varepsilon$ is compact, there exists  $t_1(\varepsilon) \in (0,1)$ such 
that the conclusion holds for any $ t \in (0, t_1(\varepsilon))$ and any $h \in B \cap \overline{\mathcal C}_\varepsilon$. But if $h \in
\overline{\mathcal C}_\varepsilon$ satisfies $||h|| > 1$, we can write it as $h = ||h|| h'$,
with $h'\in B \cap \overline{\mathcal C}_\varepsilon$. Then, for any $t \in (0,1)$,
$t^h = (t^{||h||})^{h'}$, with $0 < t^{||h||} < t$ and so the conclusion of Theorem~\ref{th:BS} 
holds for any $h \in \overline{\mathcal C}_\varepsilon$ 
provided $t \in (0,t_1(\varepsilon)]$. By Proposition~\ref{prop:cone}, the image by $\varphi$ of
$\R_{>0}^{d+1} \times(0, t_1(\varepsilon)] \times  \overline{\mathcal C}_\varepsilon$ equals
$\{ \gamma \in \R_{>0}^n \, ; \, \gamma^{m_r} \le t_0(\varepsilon)^{\varepsilon_r}, \, r=1 \dots,\ell\}$.   Note also that
${\mathcal C}_{\Delta_1,\dots, \Delta_p} = \cup_\varepsilon  \overline{\mathcal C}_\varepsilon$. 

Observe that if $\gamma=\varphi(\alpha, t,h)$, then for any $j=1, \dots, n$,
\[\gamma_{j}x^{a_j} = \alpha^{(1 a_j)} x^{a_j} = \alpha_1 t^{h_j} y^{a_j},\]
where $y_i = \alpha_{i+1} x_i$ for any $i=1, \dots, d$.
As  all $\alpha_i >0$, 
 system~\eqref{systemgamma} has the same number of positive solutions
as
 \begin{equation}\label{systemgammay}
  \sum_{j=1}^n c_{ij} \, t^{h_j} \, y^{a_j} = 0 , \;  \ i=1,\dots,d,
 \end{equation}
which is at least $p$ for $t \in (0,t_1(\varepsilon)]$.
\end{proof}

\begin{remark} {\rm Theorem~\ref{thm:main2} says that if we choose
$\varepsilon\in (0,1)^\ell$, then there exists positive real numbers
$M_1=M_1(\varepsilon),\dots,M_r=M_r(\varepsilon)$, such that  system~\eqref{systemgamma} has at least $p$ nondegenerate solutions 
in the positive orthant for any vector $\gamma$ in $\R^n_{>0}$ satisfying  \begin{equation}\label{eq:gamma}
\gamma^{m_r} < M_r \quad \mbox{for} \ r=1,\ldots,\ell.
  \end{equation}
  We have to remark that the choice of the positive constants $M_1, \dots, M_r$ is {\it not} algorithmic,
 but our result makes clear that there is an open set in coefficient space for which many positive solutions can be found, and 
 inequalities~\eqref{eq:gamma} 
 indicate ``in which directions'' the coefficients have to be scaled in order to get  at least as many positive solutions as the 
 number of decorated simplices. }

\end{remark}

As a first application of Theorems~\ref{th:BS} and~\ref{thm:main2}, we give a proof of Theorem~\ref{th:examplehhk}, 
corresponding to the example of the two component system with Hybrid Histidine Kinase (\ref{networkHHK}) in the Introduction.
 
 \begin{proof}[Proof of Theorem \ref{th:examplehhk}]
  From  $f_2=f_3=f_4=f_5=0$ we get:
 \[x_1=\dfrac{k_4k_5x_4x_5^2}{k_1k_3},\, x_2 = \dfrac{k_4k_5x_4x_5^2}{k_2k_3} + \dfrac{k_5x_4x_5}{k_2},\, x_3= \dfrac{k_5x_4x_5}{k_3},\,
 x_6=\dfrac{k_4k_5x_4x_5^2}{k_3k_6} + \dfrac{k_5x_4x_5}{k_6}.\]
 
Then, at steady state, the concentrations of the species can be obtained from the values of $x_4$ and $x_5$. If we replace these expressions 
 into the conservation laws~\eqref{eq:conservationlawsnetworkHHK}, we get the equations:
 \begin{align*}
 \dfrac{k_4k_5x_4x_5^2}{k_1k_3} + \dfrac{k_4k_5x_4x_5^2}{k_2k_3} + \dfrac{k_5x_4x_5}{k_2} + \dfrac{k_5x_4x_5}{k_3} + x_4 -  T_1 = &\, 0, \\
 x_5 + \dfrac{k_4k_5x_4x_5^2}{k_3k_6} + \dfrac{k_5x_4x_5}{k_6} - T_2 = &\, 0.
 \end{align*}
 
 We can write this system in matricial form:
 \[C \begin{pmatrix}
 x_4 & x_5 & x_4x_5 & x_4x_5^2 & 1
 \end{pmatrix}^t = 0,\]
 where $C\in \R^{2\times 5}$ is the coefficient matrix:
 \[C = \begin{pmatrix}
 1 & 0 & C_{13} & C_{14} & -T_1 \\
 0 & 1 & C_{23} & C_{24} & -T_2
 \end{pmatrix},\]
 and $C_{13}=k_5\left(\frac{1}{k_2} + \frac{1}{k_3}\right)$, $C_{14}=\frac{k_4k_5}{k_3}
 \left(\frac{1}{k_1} + \frac{1}{k_2}\right)$, $C_{23}=\frac{k_5}{k_6}$ and $C_{24}= \frac{k_4k_5}{k_3k_6}$.
If we order the variables $(x_4,x_5)$ the support of this system is:
 \[ \mathcal{A}=\{(1,0),(0,1), (1,1), (1,2),(0,0)\} .\]
 We depict in Figure \ref{fig:regulartriangulationHHK} the $2$-simplices $\Delta_1=\{(1,0), (1,1), (0,0) \}$, $\Delta_2=\{(1,1), (1,2),(0,0) \}$ and 
 $\Delta_3=\{(0,1), (1,2),(0,0) \}$, which form a 
 regular triangulation $\Gamma$ of $\mathcal{A}$, associated for instance with any height function $h:\mA \to \R$
  satisfying $h(1,0)=h_1, h(0,1)=h_2$, $h(1,1)=0$, $ h(1,2)=0$, and $h(0,0)=0$, with $h_1, h_2>0$.

 \begin{figure}[h]
 \centering
 \begin{tikzpicture}%
 	[scale=1.500000,
 	back/.style={loosely dotted, thin},
 	edge/.style={color=black, thick},
 	facet/.style={fill=red!70!white,fill opacity=0.800000},
 	vertex/.style={inner sep=1pt,circle,draw=black,fill=black,thick,anchor=base
 }]
 
 %
 %
 \coordinate (0.00000, 0.00000) at (0.00000, 0.00000);
 \coordinate (0.00000, 1.00000) at (0.00000, 1.00000);
 \coordinate (1.00000, 1.00000) at (1.00000, 1.00000);
 \coordinate (1.00000, 0.00000) at (1.00000, 0.00000);
 \coordinate (1.00000, 2.00000) at (1.00000, 2.00000);
 \fill[facet, fill=green!70!white,fill opacity=0.800000] (1.00000, 2.00000) -- (1.00000, 1.00000) -- (0.00000,
 0.00000) -- cycle {};
 \fill[facet, fill=red!50!white,fill opacity=0.800000] (1.00000, 2.00000) -- (0.00000, 1.00000) -- (0.00000,
 0.00000) -- cycle {};
 \fill[facet, fill=blue!70!white,fill opacity=0.800000] (1.00000, 0.00000) -- (1.00000, 1.00000) -- (0.00000,
 0.00000) -- cycle {};
 \draw[edge] (0.00000, 0.00000) -- (0.00000, 1.00000);
 \draw[edge] (0.00000, 0.00000) -- (1.00000, 1.00000);
 \draw[edge] (0.00000, 0.00000) -- (1.00000, 2.00000);
 \draw[edge] (0.00000, 0.00000) -- (1.00000, 0.00000);
 \draw[edge] (0.00000, 1.00000) -- (1.00000, 2.00000);
 \draw[edge] (1.00000, 0.00000) -- (1.00000, 2.00000);
 
 \node[vertex] at (0.00000, 0.00000){};
 \node[vertex] at (0.00000, 1.00000){};
 \node[vertex] at (1.00000, 0.00000){};
 \node[vertex] at (1.00000, 2.00000){};
 \node[vertex] at (1.00000, 1.00000){};
 \node at (-0.35000, 0.00000){(0,0)};
 \node at (-0.35000, 1.00000){(0,1)};
 \node at (1.35000, 0.00000){(1,0)};
 \node at (1.35000, 1.00000){(1,1)};
 \node at (1.35000, 2.00000){(1,2)};
 \end{tikzpicture}
 \hspace{1.5cm}
 \begin{tikzpicture}%
 	[x={(0.795303cm, -0.371152cm)},
 	y={(0.605733cm, 0.517961cm)},
 	z={(-0.024093cm, 0.770689cm)},
 	scale=1.300000,
 	back/.style={dotted,very thick},
 	edge/.style={color=black, thick},
 	facet/.style={fill=blue!95!black,fill opacity=0.800000},
 	vertex/.style={inner sep=1pt,circle,draw=black,fill=black,thick,anchor=base}]
 
 %
 %
 \coordinate (0.00000, 0.00000, 0.00000) at (0.00000, 0.00000, 0.00000);
 \coordinate (0.00000, 1.00000, 0.00000) at (0.00000, 1.00000, 0.00000);
 \coordinate (0.00000, 1.00000, 2.00000) at (0.00000, 1.00000, 2.00000);
 \coordinate (1.00000, 0.00000, 0.00000) at (1.00000, 0.00000, 0.00000);
 \coordinate (1.00000, 0.00000, 1.00000) at (1.00000, 0.00000, 1.00000);
 \coordinate (1.00000, 1.00000, 0.00000) at (1.00000, 1.00000, 0.00000);
 \coordinate (1.00000, 2.00000, 0.00000) at (1.00000, 2.00000, 0.00000);
 
 \node[vertex] at (0.00000, 1.00000, 0.00000)     {};
 
 \fill[facet, fill=green!70!white,fill opacity=0.800000] (1.00000, 2.00000, 0.00000) -- (1.00000, 1.00000, 0.00000) -- (0.00000,
 0.00000, 0.00000) -- cycle {};
 \fill[facet, fill=white,fill opacity=0.800000] (1.00000, 2.00000, 0.00000) -- (0.00000, 1.00000, 0.00000) -- (0.00000,
 0.00000, 0.00000) -- cycle {};
 \fill[facet, fill=white,fill opacity=0.800000] (1.00000, 0.00000, 0.00000) -- (1.00000, 1.00000, 0.00000) -- (0.00000,
 0.00000, 0.00000) -- cycle {};
 \fill[facet, fill=red!50!white,fill opacity=0.800000] (0.00000, 1.00000, 2.00000) -- (1.00000, 2.00000, 0.00000) -- (0.00000,
 0.00000, 0.00000) -- cycle {};
 \draw[edge] (0.00000, 0.00000, 0.00000) -- (1.00000, 2.00000, 0.00000);
 \fill[facet, fill=blue!70!white,fill opacity=0.800000] (1.00000, 0.00000, 1.00000) -- (1.00000, 1.00000, 0.00000) -- (0.00000,
 0.00000, 0.00000) -- cycle {};

 \draw[edge,back] (0.00000, 0.00000, 0.00000) -- (0.00000, 1.00000, 0.00000);
 \draw[edge,back] (0.00000, 1.00000, 0.00000) -- (0.00000, 1.00000, 2.00000);
 \draw[edge,back] (0.00000, 1.00000, 0.00000) -- (1.00000, 2.00000, 0.00000);
 \draw[edge,back] (1.00000, 0.00000, 0.00000) -- (1.00000, 0.00000, 1.00000);
 \draw[edge,back] (1.00000, 0.00000, 0.00000) -- (1.00000, 1.00000, 0.00000);
 \draw[edge,back] (1.00000, 0.00000, 0.00000) -- (1.00000, 0.00000, 1.00000);
 
 \draw[edge,back] (0.00000, 0.00000, 0.00000) -- (1.00000, 0.00000, 0.00000);
 \draw[edge] (0.00000, 0.00000, 0.00000) -- (1.00000, 0.00000, 1.00000);
 \draw[edge] (0.00000, 1.00000, 2.00000) -- (1.00000, 2.00000, 0.00000);
 \draw[edge] (1.00000, 1.00000, 0.00000) -- (1.00000, 2.00000, 0.00000);
 \draw[edge] (1.00000, 0.00000, 1.00000) -- (1.00000, 1.00000, 0.00000);
 \draw[edge] (0.00000, 0.00000, 0.00000) -- (1.00000, 1.00000, 0.00000);
 \draw[edge] (0.00000, 0.00000, 0.00000) -- (0.00000, 1.00000, 2.00000);
 
 \node[vertex] at (0.00000, 0.00000, 0.00000)     {};
 \node[vertex] at (0.00000, 1.00000, 2.00000)     {};
 \node[vertex] at (1.00000, 1.00000, 0.00000)     {};
 \node[vertex] at (1.00000, 0.00000, 1.00000)     {};
 \node[vertex] at (1.00000, 2.00000, 0.00000)     {};
 \node at (0.00000,-0.50000, 0.00000 )     {(0,0,0)};
 \node at (0.00000,1.1000, 2.20000 )     {(0,1,$h_2$)};
 \node at (1.50000,1.0000, 0.0000 )     {(1,1,0)};
 \node at (1.50000,2.0000, 0.0000 )     {(1,2,0)};
 \node at (-0.60000,0.0000, 1.0000 )     {(1,0,$h_1$)};
 
\draw [->,bend right=20, thick](1.00000, 0.00000, 1.05000) to (0.0000,0.1000, 1.2000 ) ;
 \end{tikzpicture}
 \caption{A regular triangulation $\Gamma$ of $\mathcal{A}$.} 
 \label{fig:regulartriangulationHHK}
\end{figure}
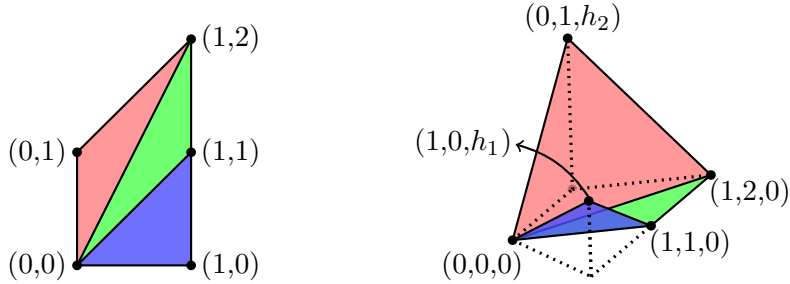

 The simplex $\Delta_1$ is positively decorated by $C$ if and only if 
 \begin{equation}\label{eq:1}
 T_1\,k_2\,k_3 - T_2\,k_2\,k_6 - T_2\,k_3\,k_6 >0,
 \end{equation} and the simplex $\Delta_3$ 
 is positively decorated by $C$ if and only if 
 \begin{equation}\label{eq:2} T_1\,k_1\,k_2 - T_2\,k_1\,k_6 -T_2\,k_2\,k_6 < 0.
 \end{equation} 
 If conditions~\eqref{eq:1} and~\eqref{eq:2} hold, then the simplex $\Delta_2$ is also positively decorated by $C$ if and only if $k_1<k_3$. 
 So, the three simplices are positively decorated by $C$ under the validity of condition (\ref{eq:conditionhhk}):
 $$k_6\left(\dfrac{1}{k_2}+\dfrac{1}{k_3}\right)< \dfrac{T_1}{T_2} <k_6\left(\dfrac{1}{k_1}+\dfrac{1}{k_2}\right).$$ 
 
Assume that both inequalities in~\eqref{eq:conditionhhk} hold.
In this case, Theorem~\ref{th:BS} says that there exists $t_0 \in \R_{>0}$ such that for all $0<t<t_0$, the system
 \vskip -10pt
 \begin{equation}\label{systemwithtej1}
 \begin{aligned}
 t^{h_1}x_4 + C_{13}\, x_4x_5 + C_{14} \, x_4x_5^2    -  T_1 &= & 0, \\
 t^{h_2}x_5  + C_{23} \, x_4x_5 + C_{24} \, x_4x_5^2 - T_2 &=& 0,
 \end{aligned}
 \end{equation}
has at least three positive nondegenerate solutions.
 
 If we make the change of variables: $\bar x_4=t^{h_1}\, x_4$, $\bar x_5=t^{h_2}\, x_5$ we have:
 \vskip -10pt
 \begin{equation}\label{systemwithtej1b}
 \begin{aligned}
 \bar x_4 + t^{-(h_1+h_2)}\,C_{13}\,\bar x_4 \bar x_5 + t^{-(h_1+2h_2)}\,C_{14}\,\bar x_4\bar x_5^2    -  T_1 &= & 0, \\
 \bar x_5  + t^{-(h_1+h_2)}\,C_{23}\,\bar x_4\bar x_5 + t^{-(h_1+2h_2)}\,C_{24}\,\bar x_4\bar x_5^2 - T_2 &=& 0.
 \end{aligned}
 \end{equation}
 
 If we consider the rescalings:
 \[
 \overline{k_4}=t^{-h_2}\,k_4,\qquad \overline{k_5}=t^{-(h_1 + h_2)}\,k_5,
 \]
 and we keep fixed the values of the remaining constants $k_1, k_2, k_3, k_6$ and
  the total concentrations $T_1$, $T_2$, then the  steady states of the dynamical system
  associated with
 the network with these rate and total conservation constants are the solutions of
 the polynomial system (\ref{systemwithtej1b}). And then, for these constants the network has at least three positive steady states. 
 If we take $N_1=t_0^{-h_2}$ and $N_2=t_0^{-h_1}$ and we consider any positive $\beta_4, \beta_5$ satisfying
 $\beta_4>N_1$ and $\frac{\beta_5}{\beta_4}>N_2$, there exist $0<t<t_0$ such that 
  $\beta_4=t^{-h_1}$ and $\beta_5=t^{-(h_1+h_2)}$ and we are done.
 
 Another way to finish the proof of Theorem~\ref{th:examplehhk} is using Theorem~\ref{thm:main2}.
 The inequalities that define the cone $\mathcal{C}_{\Gamma}$ are: $\langle m_1,h \rangle > 0$, 
 $\langle m_2,h \rangle > 0$, where $m_1=(1,0,-2,1,0)$ and $m_2=(0,1,1,-1,-1)$.  
 Fix $\varepsilon\in (0,1)^2$. As (\ref{eq:conditionhhk}) holds, Theorem~\ref{thm:main2} says that there exist
  $M_1=M_1(\varepsilon), M_2=M_2(\varepsilon) > 0$ such that the polynomial system 
 \vskip -10pt
 \begin{equation}\label{systemwithgamma1}
 \begin{aligned}
 \gamma_1\,x_4 + \gamma_3\,C_{13}\,x_4x_5 + \gamma_4\,C_{14}\,x_4x_5^2    -  \gamma_5\,T_1 &= & 0, \\
 \gamma_2\,x_5  + \gamma_3\,C_{23}\,x_4x_5 + \gamma_4\,C_{24}\,x_4x_5^2 - \gamma_5\,T_2 &=& 0,
 \end{aligned}
 \end{equation}
 has at least three nondegenerate positive solutions for any vector $\gamma\in (\R_{>0})^5$ 
 satisfying $\gamma^{m_1}<M_1$ and $\gamma^{m_2}<M_2$. 
 In particular, this holds if we take $\gamma_1=\gamma_2=\gamma_5=1$, and $\gamma_3$ and $\gamma_4$ satisfy:
 \begin{equation}\label{eq:3}
 \gamma_3^{-2}\gamma_4 < M_1,\qquad \gamma_3\gamma_4^{-1} < M_2.\end{equation}
  If we call $\beta_4=\frac{\gamma_4}{\gamma_5}$, $\beta_5=\gamma_5$,  
  $N_1=\frac{1}{M_1}$, $N_2=\frac{1}{M_2}$, the inequalities in~\eqref{eq:3}
   are equivalent to $\beta_4 > N_1$ and 
 $\frac{\beta_5}{\beta_4} > N_2$. Then, if $\beta_4$ and $\beta_5$ satisfy these bounds, rescaling of the given parameters 
 $k_4, k_5$ by $\overline{k_4}=\beta_4 \, k_4, \overline{k_5}=\beta_5 \, k_5$, gives raise to a 
 multistationary dynamical system, as we claimed.
 
 \end{proof}

\section{The mixed approach}\label{sec:2}
In this section, we present a similar but different approach to Theorems~\ref{th:BS} and~\ref{thm:main2}. 
As the polynomials  $f_1, \dots, f_d$ 
might have different supports $\mathcal A_1, \dots, \mathcal A_d \subset \Z^d$,
 one usually takes the union of the supports $\mathcal A = \cup_{i=1}^d \mA_i$.
That is, we can write the polynomial system
 \begin{equation}\label{systemwithtcayley}
 f_{i}(x)=\sum_{a_j\in \mA_i} c_{i,j}\, x^{a_j} \in \R[x_1,\dots,x_d], \ i=1,\dots,d
 \end{equation} 
in the form
 \begin{equation*}
 f_{i}(x)=\sum_{a_j \in \mA}  c_{i,j}\, x^{a_j} \in \R[x_1,\dots,x_d], \ i=1,\dots,d, 
 \end{equation*}
where $c_{i,j}=0$ in case $a_j \notin \mathcal A_i$.

If one takes the same height function from $\mA$ to $\R$ for each equation,  the coefficient matrix 
$C$ might have many zero minors. Thus, we now allow
different height functions $h^{(i)}: \mA_i \to \R$, $i=1,\ldots,d$. Instead of considering regular subdivisions of $\mA$, 
we will consider \textit{regular mixed subdivisions} of the  Minkowski sum ${\mathcal M}=\sum_{i=1}^d \mA_i$
defined by height functions $h^{(i)}: \mA_i \to \R$, $i=1,\ldots,d$.
The projection of the lifted points in each of the faces of the lower convex hull of the Minkowski sum 
$\sum_{i=1}^d {\mA}^{h^{(i)}}$  of the lifted point sets ${\mA}^{h^{(i)}} \subset \R^{d+1}$ defines the associated regular mixed 
subdivision $S_h$ of $\mathcal M$.  The convex hull of the cells in $S_h$ do not intersect or the 
intersection is a common face.  
 
Regular mixed subdivisions of ${\mathcal M}$ 
 are in bijection with regular subdivisions of the associated \textit{Cayley configuration}  $C(\mA_1, \dots, \mA_d) $.
This is  the lattice configuration in $\Z^{d}\times\Z^{d}$ defined by
 \begin{equation}\label{E:Cayley}
 C({\mathcal A_1},\ldots,\mathcal A_d)\, = \, (\mathcal A_1 \times \{e_1\}) \cup  \cdots
 ({\mathcal A_{d-1}} \times \{e_{d-1}\}) \cup ({\mathcal A_d} \times \{e_d\}),
 \end{equation}
 where $e_1,\dots,e_d$ denotes the canonical basis in $\Z^{d}$. 
 This is the support of the Cayley polynomial: 
 \begin{equation*}
 F(x,y)=\sum_{i=1}^{d} y_if_i(x), 
 \end{equation*}
 in variables 
 $(x_1,\ldots,x_d,y_1,\ldots,y_{d})$,
 associated with polynomials $f_i(x)$ with support in $\mathcal A_i$, $i=1,\dots,d$.
 
 Note that the
 sum of the last $d$ coordinates of any point in $ C({\mathcal A_1},\ldots,\mathcal A_d)$ equals $1$, 
 so the maximal dimension of a simplex in the Cayley configuration is $2d-1$ and then this simplex consists of $2d$ points.
 We will assume that $ C({\mathcal A_1},\ldots,\mathcal A_d)$  contains a $(2d-1)$-simplex.
  
A tuple of height functions $(h^{(1)}, \dots, h^{(d)})$ as above can be identified with
a height function $h: C({\mathcal A_1},\ldots,\mathcal A_d) \to \R$,
defining $h(a_j, e_i) = h^{(i)} (a_j), i=1, \dots, d$.  In case $\Delta$  is a $(2d-1)$-simplex in the associated regular 
subdivision $\Gamma_h$ of $C({\mathcal A_1},\ldots,\mathcal A_d)$, necessarily $\Delta$ contains at least one point $(a_j, e_i)$
in each $\mA_i$.
The corresponding maximal cell in the associated regular subdivision $S_h$ 
of $\mathcal M$ consists of all points of the form $b_1+\dots +b_d$ with $(b_i,e_i)$ in $\Delta$. 
For more details about the translation between regular subdivisions of $C(\mA_1, \dots, \mA_d)$ and 
regular mixed subdivisions of $\mathcal M$, we refer to Section 9.2 in \cite{triangulations}. 
We show this correspondence in Example~\ref{ej:mixed} below.
 
 \begin{defi}  A $(2d-1)$-simplex $\Delta$ in the Cayley configuration $C({\mathcal A_1},\ldots,\mathcal A_d)$ 
 is said to be mixed if it consists of two points $(a_{j_1},e_i),(a_{j_2},e_i)$ for each $i=1,\dots,d$, 
 with $a_{j_1},a_{j_2} \in \mathcal{A}_i$. A mixed simplex $\Delta$ is said to be
positively decorated by $C$ if for each $i=1,\dots, d$, the coefficients of the polynomial $f_i$ corresponding to the monomials $a_{j_1}$ 
and $a_{j_2}$ have different signs, that is, if $c_{i,j_1}c_{i,j_2}<0$.
 \end{defi}

Let  $\Gamma$ be a regular subdivision of the Cayley configuration  $C({\mathcal A_1},\ldots,\mathcal A_d)$. Let $h$ be a height vector
that induces $\Gamma$ and denote by $h^{(1)},\dots,h^{(d)}$ the real vectors of size equal to the cardinality of
$\mathcal A_i$, such that $h^{(i)}(a_j)=h(a_j, e_i)$, for $i=1,\dots,d$, and $a_j\in\mathcal{A}_i$.
Consider the family of polynomial systems parametrized by a positive real  number $t$:
\begin{equation} \label{eq:mixed}
 f_{i,t}(x)=\sum_{a_j \in \mA_i} c_{i,j}\, t^{h^{(i)}(a_j)} \, x^{a_j} \in \R[x_1,\dots,x_d], \ i=1,\dots,d, \ t>0.
 \end{equation}

We then have:

 \begin{thm}\label{thm:mixed} Let $\mathcal{A}_1,\dots, \mathcal{A}_d$ be finite sets in $\Z^d$. 
 Assume there are $p$ mixed $(2d-1)$-simplices $\Delta_1, \dots, \Delta_p$
 which occur in a regular subdivision $\Gamma$ of  $C({\mathcal A_1},\ldots,\mathcal A_d)$ and which are positively decorated by a 
 matrix $C  \in \R^{d \times n}$. Let $h$ be a height function inducing $\Gamma$ and $h^{(i)}, i=1,\dots,d$, defined as before. 
Then, there exists $t_0\in\R_{>0}$ 
 such that for all $0<t<t_0$, the number of (nondegenerate) solutions of (\ref{eq:mixed}) 
contained in the positive orthant is at least $p$.
  In particular, the result holds if $\Delta_1, \Delta_2$ are two mixed $(2d-1)$-simplices of $C({\mathcal A_1},\ldots,\mathcal A_d)$ which share a facet.
 \end{thm}
 
 \begin{proof} Let $\Delta$ be a mixed $(2d-1)$-simplex in $\Gamma$. Then it consists of $2d$ points: two points $(a_{j_1},e_i),(a_{j_2},e_i)$ 
 for each $i=1,\dots,d$, with $a_{j_1}, a_{j_2} \in \mathcal{A}_i$. Consider the system (\ref{systemwithtcayley})
  restricted to the binomials with exponents  $a_{j_1},a_{j_2}$ in each $f_i$. 
 When $\Delta$ is positively decorated by $C$, we get a binomial system of equations equal to zero with coefficients of opposite signs:
 \[  c_{i,j_1} \, x^{a_{j_1}} +  c_{i,j_2} \,x^{a_{j_2}} \, = \, 0, \quad i=1, \dots, d.\]
 The positive solutions of this  binomial system are in correspondence with the solutions of a system of a form: 
 \[x^M=\beta,\]
 where $M\in \R^{d\times d}$ is the matrix with $i$-th 
 row equal to $a_{j_1}-a_{j_2}$ for each $i=1,\dots,d$, and $\beta = - \frac{c_{i,j_2}}{c_{i,j_1}} \in \R^d_{>0}$. Taking logarithms, we obtain the equivalent linear
system:
\begin{equation}\label{eq:M}
M^t\log(x) = \log(\beta),
\end{equation} 
where $\log(x)=(\log(x_1),\dots,\log(x_d))$. As $\Delta$ is a maximal simplex, the matrix $M$ is invertible. 
Then, the linear system~\eqref{eq:M} has a solution, and thus the binomial system has a positive solution.
Therefore, for each positively decorated simplex $\Delta$ with vertices $(a_{j_1},e_i),(a_{j_2},e_i)$, for each $i$, 
system (\ref{systemwithtcayley}) restricted to the monomials with exponents $a_{j_1}$,$a_{j_2}$ in each $f_i$ has a solution in $\R_{>0}^d$. 
The rest of the proof follows from the arguments in the proof of Theorem 3.4 in \cite{bihan}.
  \end{proof}

Furthermore, with  a similar proof as Theorem \ref{thm:main2}, we have:

 \begin{thm}\label{thm:mixedgamma}
Let $\mathcal{A}_1,\dots, \mathcal{A}_d$ be finite sets in $\Z^d$. Assume  there exist $p$ mixed $(2d-1)$-simplices $\Delta_1, \dots, \Delta_p$  in
$C({\mathcal A_1},\ldots,\mathcal A_d)$, which are part of a regular subdivision  (for instance, when $p=2$ and the two simplices share a facet)
and are positively decorated by $C$.
 Set $N=|\mathcal A_1|+\dots |\mathcal A_d|$. Assume that the cone ${\mathcal C}_{\Delta_1,\dots, \Delta_p}$ of all
 height vectors $h$ inducing regular subdivisions of $C({\mathcal A_1},\ldots,\mathcal A_d)$ containing $\Delta_1, \dots, \Delta_p$ is defined by
 \begin{equation} \label{E:coneformixed} {\mathcal C}_{\Delta_1,\dots, \Delta_p} \, = \, \{ h \in \R^N \, : \, 
 \langle m_r, h \rangle  > 0 , \; r=1,\ldots,\ell\},
 \end{equation}
 where $m_r=(m_{r,1},\ldots,m_{r,N}) \in \R^N$.

Then, for any $\varepsilon \in (0,1)^\ell$ there exists $t_0(\varepsilon) >0$ such that for any $\gamma$ in the set
\[ U \, = \, \cup_{\varepsilon \in (0,1)^\ell} \,  \{ \gamma=(\gamma^1,\dots,\gamma^d) \in \R_{>0}^N, ;
 \, \gamma^{m_r} \le t_0(\varepsilon)^{\varepsilon_r}, \, r=1 \dots,\ell\},\] the system 
 \begin{equation*}
 \sum_{a_j\in \mA_i} c_{ij}\gamma^i_{j}\, x^{a_j}=0 , \;  \ i=1,\dots,d,
 \end{equation*}
 has at least $p$ nondegenerate solutions in the positive orthant, where $\gamma^i$ is a vector of size 
 $|\mA_i|$ with coordinates $\gamma^i_{j}$, with $a_j\in \mA_i$.
  \end{thm}

%
%
%

%
%
%



We now present an application of the mixed approach in Theorem~\ref{thm:mixedgamma} 
to the previous example of the two component system with Hybrid Histidine Kinase \eqref{networkHHK}.

\begin{example}\label{ej:mixed}
\rm
 
Recall that we are looking for positive solutions of the system:
 $$C \begin{pmatrix}
 x_4 & x_5 & x_4x_5 & x_4x_5^2 & 1
 \end{pmatrix}^t = 0,$$
 where $C\in \R^{2\times 5}$ is the coefficient matrix:
 $$C = \begin{pmatrix}
 1 & 0 & C_{13} & C_{14} & -T_1 \\
 0 & 1 & C_{23} & C_{24} & -T_2
 \end{pmatrix},$$
 with $C_{13}=k_5\left(\frac{1}{k_2} + \frac{1}{k_3}\right)$, $C_{14}= \frac{k_4k_5}{k_3}
 \left(\frac{1}{k_1} + \frac{1}{k_2}\right)$, $C_{23}=\frac{k_5}{k_6}$ and $C_{24}=\frac{k_4k_5}{k_3k_6}$.
 
 The supports of the first and second polynomials are
  $\mathcal{A}_1=\{(1,0),(1,1), (1,2),(0,0)\}$ and $\mathcal{A}_2=\{(0,1), (1,1), (1,2),(0,0)\}$ respectively. 
  We want to find mixed positively decorated mixed  $3$-simplices of the Cayley configuration 
  $C({\mathcal A_1},\mathcal A_2)$ occuring  in a regular subdivision. As we mentioned, these  mixed $3$-simplices correspond to 
 maximal dimension $2$ mixed cells of the associated mixed subdivision of the Minkowski sum 
 $\mathcal{A}_1+\mathcal{A}_2$ (Fig. \ref{fig:Minkowskisum}).
  
  \begin{figure}[h]
  	
  	\centering
  	\begin{tikzpicture}
  	[scale=0.800000,
  	back/.style={loosely dotted, thin},
  	edge/.style={color=black, thick},
  	facet/.style={fill=red!70!white,fill opacity=0.800000},
  	vertex/.style={inner sep=1.2pt,circle,draw=black,fill=black,thick,anchor=base
  	}]
  	%
  	%
  	\coordinate (0.00000, 0.00000) at (0.00000, 0.00000);
  	\coordinate (1.00000, 1.00000) at (1.00000, 1.00000);
  	\coordinate (1.00000, 0.00000) at (1.00000, 0.00000);
  	\coordinate (1.00000, 2.00000) at (1.00000, 2.00000);
  	
  	\draw[edge] (0.00000, 0.00000) -- (1.00000, 2.00000);
  	\draw[edge] (0.00000, 0.00000) -- (1.00000, 0.00000);
  	\draw[edge] (1.00000, 0.00000) -- (1.00000, 2.00000);
  	
  	\node[vertex] at (0.00000, 0.00000){};
  	\node[vertex] at (1.00000, 0.00000){};
  	\node[vertex] at (1.00000, 2.00000){};
  	\node[vertex] at (1.00000, 1.00000){};
  	\node at (2.50000, 1.000000) { $+$};
  	\node at (0.50000, -0.600000) { $\mathcal{A}_1$};
  	\node at (0.50000, -1.0000){};
  	\end{tikzpicture}
  	\hspace{0.6cm}
  	\begin{tikzpicture}%
  	[scale=0.800000,
  	back/.style={loosely dotted, thin},
  	edge/.style={color=black, thick},
  	facet/.style={fill=red!70!white,fill opacity=0.800000},
  	vertex/.style={inner sep=1.2pt,circle,draw=black,fill=black,thick,anchor=base
  	}]
  	%
  	%
  	\coordinate (0.00000, 0.00000) at (0.00000, 0.00000);
  	\coordinate (1.00000, 1.00000) at (1.00000, 1.00000);
  	\coordinate (0.00000, 1.00000) at (0.00000, 1.00000);
  	\coordinate (1.00000, 2.00000) at (1.00000, 2.00000);
  	
  	\draw[edge] (0.00000, 0.00000) -- (1.00000, 1.00000);
  	\draw[edge] (0.00000, 0.00000) -- (0.00000, 1.00000);
  	\draw[edge] (0.00000, 1.00000) -- (1.00000, 2.00000);
  	\draw[edge] (1.00000, 1.00000) -- (1.00000, 2.00000);
  	
  	\node[vertex] at (0.00000, 0.00000){};
  	\node[vertex] at (0.00000, 1.00000){};
  	\node[vertex] at (1.00000, 2.00000){};
  	\node[vertex] at (1.00000, 1.00000){}; 
  	\node at (2.50000, 1.00000){$=$};
  	\node at (0.50000, -0.60000){{ $\mathcal{A}_2$}};
  	\node at (0.50000, -1.0000){};
  	\end{tikzpicture}
  	\hspace{0.3cm}
  	\begin{tikzpicture}%
  	[scale=0.800000,
  	back/.style={loosely dotted, thin},
  	edge/.style={color=black, thick},
  	facet/.style={fill=red!70!white,fill opacity=0.800000},
  	vertex/.style={inner sep=1.2pt,circle,draw=black,fill=black,thick,anchor=base
  	}]
  	%
  	%
  	\coordinate (0.00000, 0.00000) at (0.00000, 0.00000);
  	\coordinate (0.00000, 1.00000) at (0.00000, 1.00000);
  	\coordinate (1.00000, 1.00000) at (1.00000, 1.00000);
  	\coordinate (1.00000, 0.00000) at (1.00000, 0.00000);
  	\coordinate (1.00000, 2.00000) at (1.00000, 2.00000);
  	\coordinate (1.00000, 3.00000) at (1.00000, 3.00000);
  	\coordinate (2.00000, 1.00000) at (2.00000, 1.00000);
  	\coordinate (2.00000, 2.00000) at (2.00000, 2.00000);
  	\coordinate (2.00000, 3.00000) at (2.00000, 3.00000);
  	\coordinate (2.00000, 4.00000) at (2.00000, 4.00000);
  	\draw[edge] (0.00000, 0.00000) -- (0.00000, 1.00000);
  	\draw[edge] (0.00000, 0.00000) -- (1.00000, 0.00000);
  	\draw[edge] (0.00000, 1.00000) -- (1.00000, 3.00000);
  	\draw[edge] (1.00000, 0.00000) -- (2.00000, 1.00000);
  	\draw[edge] (2.00000, 1.00000) -- (2.00000, 4.00000);
  	\draw[edge] (1.00000, 3.00000) -- (2.00000, 4.00000);
  	
  	\node[vertex] at (0.00000, 0.00000){};
  	\node[vertex] at (0.00000, 1.00000){};
  	\node[vertex] at (1.00000, 0.00000){};
  	\node[vertex] at (1.00000, 2.00000){};
  	\node[vertex] at (1.00000, 1.00000){};
  	\node[vertex] at (1.00000, 3.00000){};
  	\node[vertex]  at (2.00000, 1.00000){};
  	\node[vertex]  at (2.00000, 2.00000){};
  	\node[vertex] at (2.00000, 3.00000){};
  	\node[vertex]  at (2.00000, 4.00000){};
  	\end{tikzpicture}
  	\caption{The Minkowski sum  $\mathcal M =\mathcal{A}_1 + \mathcal{A}_2$.} \label{fig:Minkowskisum}
  \end{figure}
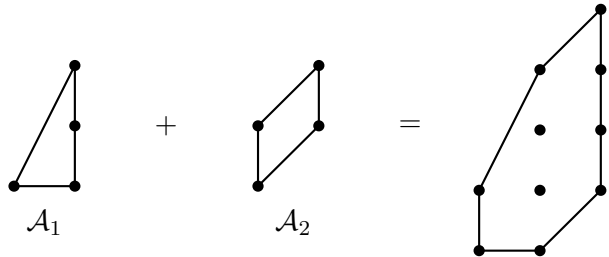
 
 We can choose the following mixed $3$-simplices with vertices in $C({\mathcal A_1},\mathcal A_2)$: 
 \[\Delta_1=\{(0,0,e_1),(1,2,e_1),(0,0,e_2),(0,1,e_2)\},\] \[\Delta_2=
 \{(0,0,e_1),(1,2,e_1),(0,0,e_2),(1,1,e_2)\},\]   \[\Delta_3=\{(0,0,e_1),(1,0,e_1),(0,0,e_2),(1,1,e_2)\},\]
which are positively decorated by $C$. These simplices 
$\Delta_1$, $\Delta_2$ and $\Delta_3$ are in correspondence, respectively, with the mixed cells
 $\sigma_1=\{((0,0) = (0,0)+ (0,0), (1,2)= (1,2)+(0,0), (0,1)=(0,0)+(0,1), (1,3)=(1,2)+(0,1))\},$
$\sigma_2=\{((0,0), (1,2), (1,1), (2,3))\}$ and $\sigma_3=\{((0,0), (1,0), (1,1),(2,1))\}$,  depicted in Figure \ref{fig:mixedsubdivisions}.

   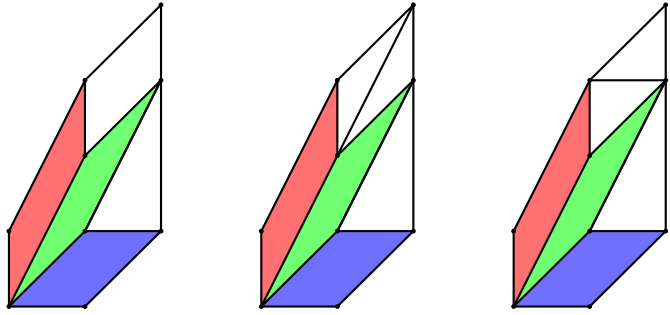
\begin{figure}[h]
 	\centering
  \begin{tikzpicture}%
 	[scale=1.000000,
 	back/.style={loosely dotted, thin},
 	edge/.style={color=black, thick},
 	facet/.style={fill=red!70!white,fill opacity=0.800000},
 	vertex/.style={inner sep=0.2pt,circle,draw=black,fill=black,thick,anchor=base
 }]
 %
 %
 \coordinate (0.00000, 0.00000) at (0.00000, 0.00000);
 \coordinate (0.00000, 1.00000) at (0.00000, 1.00000);
 \coordinate (1.00000, 1.00000) at (1.00000, 1.00000);
 \coordinate (1.00000, 0.00000) at (1.00000, 0.00000);
 \coordinate (1.00000, 2.00000) at (1.00000, 2.00000);
 \coordinate (1.00000, 3.00000) at (1.00000, 3.00000);
 \coordinate (2.00000, 1.00000) at (2.00000, 1.00000);
 \coordinate (2.00000, 2.00000) at (2.00000, 2.00000);
 \coordinate (2.00000, 3.00000) at (2.00000, 3.00000);
 \coordinate (2.00000, 4.00000) at (2.00000, 4.00000);
 \fill[facet, fill=green!70!white,fill opacity=0.800000] (1.00000, 2.00000) --(2.00000, 3.00000) -- (1.00000, 1.00000) -- (0.00000,
 0.00000) -- cycle {};
 \fill[facet, fill=red!70!white,fill opacity=0.800000] (1.00000, 2.00000) -- (1.00000, 3.00000) -- (0.00000, 1.00000) -- (0.00000,
 0.00000) -- cycle {};
 \fill[facet, fill=blue!70!white,fill opacity=0.800000] (1.00000, 0.00000) -- (2.00000, 1.00000) -- (1.00000, 1.00000) -- (0.00000,
 0.00000) -- cycle {};
 \draw[edge] (0.00000, 0.00000) -- (0.00000, 1.00000);
 \draw[edge] (0.00000, 0.00000) -- (1.00000, 1.00000);
 \draw[edge] (0.00000, 0.00000) -- (1.00000, 2.00000);
 \draw[edge] (0.00000, 0.00000) -- (1.00000, 0.00000);
 \draw[edge] (0.00000, 1.00000) -- (1.00000, 3.00000);
 \draw[edge] (1.00000, 0.00000) -- (2.00000, 1.00000);
 \draw[edge] (1.00000, 1.00000) -- (2.00000, 1.00000);
 \draw[edge] (1.00000, 1.00000) -- (2.00000, 3.00000);
 \draw[edge] (1.00000, 1.00000) -- (2.00000, 3.00000);
\draw[edge] (1.00000, 2.00000) -- (2.00000, 3.00000);
  \draw[edge] (2.00000, 1.00000) -- (2.00000, 4.00000);
    \draw[edge] (1.00000, 3.00000) -- (2.00000, 4.00000);
    \draw[edge] (1.00000, 2.00000) -- (1.00000, 3.00000);
 
 \node[vertex] at (0.00000, 0.00000){};
 \node[vertex] at (0.00000, 1.00000){};
 \node[vertex] at (1.00000, 0.00000){};
 \node[vertex] at (1.00000, 2.00000){};
 \node[vertex] at (1.00000, 1.00000){};
 \node[vertex] at (2.00000, 1.00000){};
 \node[vertex] at (2.00000, 3.00000){};
 \node[vertex] at (1.00000, 3.00000){};
 \node[vertex] at (2.00000, 4.00000){};
 \end{tikzpicture}
 \hspace{1.0cm}
 \begin{tikzpicture}%
 [scale=1.000000,
 back/.style={loosely dotted, thin},
 edge/.style={color=black, thick},
 facet/.style={fill=red!70!white,fill opacity=0.800000},
 vertex/.style={inner sep=0.2pt,circle,draw=black,fill=black,thick,anchor=base
 }]
 %
 %
 \coordinate (0.00000, 0.00000) at (0.00000, 0.00000);
 \coordinate (0.00000, 1.00000) at (0.00000, 1.00000);
 \coordinate (1.00000, 1.00000) at (1.00000, 1.00000);
 \coordinate (1.00000, 0.00000) at (1.00000, 0.00000);
 \coordinate (1.00000, 2.00000) at (1.00000, 2.00000);
 \coordinate (1.00000, 3.00000) at (1.00000, 3.00000);
 \coordinate (2.00000, 1.00000) at (2.00000, 1.00000);
 \coordinate (2.00000, 2.00000) at (2.00000, 2.00000);
 \coordinate (2.00000, 3.00000) at (2.00000, 3.00000);
 \coordinate (2.00000, 4.00000) at (2.00000, 4.00000);
 \fill[facet, fill=green!70!white,fill opacity=0.800000] (1.00000, 2.00000) --(2.00000, 3.00000) -- (1.00000, 1.00000) -- (0.00000,
 0.00000) -- cycle {};
 \fill[facet, fill=red!70!white,fill opacity=0.800000] (1.00000, 2.00000) -- (1.00000, 3.00000) -- (0.00000, 1.00000) -- (0.00000,
 0.00000) -- cycle {};
 \fill[facet, fill=blue!70!white,fill opacity=0.800000] (1.00000, 0.00000) -- (2.00000, 1.00000) -- (1.00000, 1.00000) -- (0.00000,
 0.00000) -- cycle {};
 \draw[edge] (0.00000, 0.00000) -- (0.00000, 1.00000);
 \draw[edge] (0.00000, 0.00000) -- (1.00000, 1.00000);
 \draw[edge] (0.00000, 0.00000) -- (1.00000, 2.00000);
 \draw[edge] (0.00000, 0.00000) -- (1.00000, 0.00000);
 \draw[edge] (0.00000, 1.00000) -- (1.00000, 3.00000);
 \draw[edge] (1.00000, 0.00000) -- (2.00000, 1.00000);
 \draw[edge] (1.00000, 1.00000) -- (2.00000, 1.00000);
 \draw[edge] (1.00000, 1.00000) -- (2.00000, 3.00000);
 \draw[edge] (1.00000, 1.00000) -- (2.00000, 3.00000);
 \draw[edge] (1.00000, 2.00000) -- (2.00000, 3.00000);
 \draw[edge] (2.00000, 1.00000) -- (2.00000, 4.00000);
 \draw[edge] (1.00000, 3.00000) -- (2.00000, 4.00000);
 \draw[edge] (1.00000, 2.00000) -- (1.00000, 3.00000);
  \draw[edge] (1.00000, 2.00000) -- (2.00000, 4.00000);
 
 \node[vertex] at (0.00000, 0.00000){};
 \node[vertex] at (0.00000, 1.00000){};
 \node[vertex] at (1.00000, 0.00000){};
 \node[vertex] at (1.00000, 2.00000){};
 \node[vertex] at (1.00000, 1.00000){};
 \node[vertex] at (2.00000, 1.00000){};
 \node[vertex] at (2.00000, 3.00000){};
 \node[vertex] at (1.00000, 3.00000){};
 \node[vertex] at (2.00000, 4.00000){};
 \end{tikzpicture}
  \hspace{1.0cm}
  \begin{tikzpicture}%
 [scale=1.000000,
 back/.style={loosely dotted, thin},
 edge/.style={color=black, thick},
 facet/.style={fill=red!70!white,fill opacity=0.800000},
 vertex/.style={inner sep=0.2pt,circle,draw=black,fill=black,thick,anchor=base
 }]
 %
 %
 \coordinate (0.00000, 0.00000) at (0.00000, 0.00000);
 \coordinate (0.00000, 1.00000) at (0.00000, 1.00000);
 \coordinate (1.00000, 1.00000) at (1.00000, 1.00000);
 \coordinate (1.00000, 0.00000) at (1.00000, 0.00000);
 \coordinate (1.00000, 2.00000) at (1.00000, 2.00000);
 \coordinate (1.00000, 3.00000) at (1.00000, 3.00000);
 \coordinate (2.00000, 1.00000) at (2.00000, 1.00000);
 \coordinate (2.00000, 2.00000) at (2.00000, 2.00000);
 \coordinate (2.00000, 3.00000) at (2.00000, 3.00000);
 \coordinate (2.00000, 4.00000) at (2.00000, 4.00000);
 \fill[facet, fill=green!70!white,fill opacity=0.800000] (1.00000, 2.00000) --(2.00000, 3.00000) -- (1.00000, 1.00000) -- (0.00000,
 0.00000) -- cycle {};
 \fill[facet, fill=red!70!white,fill opacity=0.800000] (1.00000, 2.00000) -- (1.00000, 3.00000) -- (0.00000, 1.00000) -- (0.00000,
 0.00000) -- cycle {};
 \fill[facet, fill=blue!70!white,fill opacity=0.800000] (1.00000, 0.00000) -- (2.00000, 1.00000) -- (1.00000, 1.00000) -- (0.00000,
 0.00000) -- cycle {};
 \draw[edge] (0.00000, 0.00000) -- (0.00000, 1.00000);
 \draw[edge] (0.00000, 0.00000) -- (1.00000, 1.00000);
 \draw[edge] (0.00000, 0.00000) -- (1.00000, 2.00000);
 \draw[edge] (0.00000, 0.00000) -- (1.00000, 0.00000);
 \draw[edge] (0.00000, 1.00000) -- (1.00000, 3.00000);
 \draw[edge] (1.00000, 0.00000) -- (2.00000, 1.00000);
 \draw[edge] (1.00000, 1.00000) -- (2.00000, 1.00000);
 \draw[edge] (1.00000, 1.00000) -- (2.00000, 3.00000);
 \draw[edge] (1.00000, 1.00000) -- (2.00000, 3.00000);
 \draw[edge] (1.00000, 2.00000) -- (2.00000, 3.00000);
 \draw[edge] (2.00000, 1.00000) -- (2.00000, 4.00000);
 \draw[edge] (1.00000, 3.00000) -- (2.00000, 4.00000);
 \draw[edge] (1.00000, 2.00000) -- (1.00000, 3.00000);
 \draw[edge] (2.00000, 3.00000) -- (1.00000, 3.00000);
 
 \node[vertex] at (0.00000, 0.00000){};
 \node[vertex] at (0.00000, 1.00000){};
 \node[vertex] at (1.00000, 0.00000){};
 \node[vertex] at (1.00000, 2.00000){};
 \node[vertex] at (1.00000, 1.00000){};
 \node[vertex] at (2.00000, 1.00000){};
 \node[vertex] at (2.00000, 3.00000){};
 \node[vertex] at (1.00000, 3.00000){};
 \node[vertex] at (2.00000, 4.00000){};
 \end{tikzpicture}
\caption{Three regular mixed subdivisions of  $\mathcal{M}$
 that contain the mixed cells $\sigma_1$, $\sigma_2$ and $\sigma_3$.} \label{fig:mixedsubdivisions}
\end{figure}
The cone $\mathcal{C}_{\Delta_1,\Delta_2,\Delta_3}$ of height vectors $h=(h_1,\dots,h_8)\in\R^{8}$ 
inducing regular subdivisions of $C({\mathcal A_1},\mathcal A_2)$ 
containing $\Delta_1$, $\Delta_2$ and $\Delta_3$ is defined by the inequalities $\langle m_i, h \rangle  > 0$, $i=1,\dots,8$, where
\begin{align*}
&m_1=(1,0,-1,0,2,0,0,-2),\ m_2=(0,1,-1,0,1,0,0,-1),\ m_3=(0,0,-1,1,0,0,1,-1),\\
&m_4=(0,0,-1,1,1,1,0,-2),\ m_5=(1,0,1,-2,0,-2,0,2),\ m_6=(0,1,0,-1,0,-1,0,1),\\
&m_7=(1,0,0,-1,1,-1,0,0),\ m_8=(1,0,0,-1,0,-2,1,1),
\end{align*}
with $h_1=h(1,0,e_1)$, $h_2=h(1,1,e_1)$, $h_3=h(1,2,e_1)$, $h_4=h(0,0,e_1)$, 
$h_5=h(0,1,e_2)$, $h_6=h(1,1,e_2)$, $h_7=h(1,2,e_2)$ and $h_8=h(0,0,e_2)$.

Fix $\varepsilon\in(0,1)^8$. Theorem $\ref{thm:mixedgamma}$ says that there exist 
positive constants $M_i=M_i(\varepsilon)$, $i=1,\dots,8$ such that the number of positive nondegenerate solutions of the polynomial system
 \vskip -10pt
 \begin{equation}\label{systemwithgammamixed}
 \begin{aligned}
 \gamma^1_1\,x_4 + \gamma^1_2\,C_{13}\,x_4x_5 + \gamma^1_3\,C_{14}\,x_4x_5^2    -  \gamma^1_4\,T_1 &=  0, \\
 \gamma^2_1\,x_5  + \gamma^2_2\,C_{23}\,x_4x_5 + \gamma^2_3\,C_{24}\,x_4x_5^2 - \gamma^2_4\,T_2 &= 0,
 \end{aligned}
 \end{equation}
 is at least the number of mixed positively decorated simplices, in this case $3$, for any vector $\gamma=(\gamma^1_1, \gamma^1_2, \gamma^1_3,
 \gamma^1_4, \gamma^2_1, \gamma^2_2, \gamma^2_3, \gamma^2_4)\in \R_{>0}^8$ that satisfies $\gamma^{m_i}<M_i$, for each $i=1,\dots,8$. 
\end{example}
 
In particular we have the following result:

\begin{prop} \label{prop:mixedHK}
Given positive reactions constants $k_1,\dots,k_6$ and positive total conservations constants $T_1$ and $T_2$, there exist positive 
constants $N_1$, $N_2$, $N_3$ and $N_4$ such that for any $\beta_1, \beta_2 >0$ satisfying \[N_1<\beta_1,\quad N_2<\beta_2,
\quad \frac{\beta_2}{\beta_1}<N_3,\quad \frac{\beta_1}{(\beta_2)^2}<N_4,\] the dynamical system corresponding to the  Hybrid Histidine Kinase network (\ref{networkHHK}) 
has at least $3$ positive steady states, after replacing $k_1$ by $\bar k_1=
(\beta_1(\frac{1}{k_1+k_2})-\frac{1}{k_2})^{-1}$ and rescaling $\bar k_6= (\beta_2)^{-1}\, k_6$, 
without altering the value of the other reaction and total conservation constants.     \end{prop}

\begin{proof}
Take any positive vector $\gamma=(\gamma^1_1, \gamma^1_2, \gamma^1_3, \gamma^1_4, \gamma^2_1, 
\gamma^2_2, \gamma^2_3,\gamma^2_4)$
satisfying  $\gamma^1_1 =\gamma^1_2=\gamma^1_4, =\gamma^2_1=\gamma^2_4=1$, and $\gamma^2_2=\gamma^2_3$. Call
 $\beta_1=\gamma^1_3$ and  $\beta_2=\gamma^2_2$. Then if $\beta_1$, $\beta_2$ satisfy
\[N_1<\beta_1,\quad N_2<\beta_2,\quad \frac{\beta_2}{\beta_1}<N_3, \quad \frac{\beta_1}{(\beta_2)^2}<N_4,\]
where $N_1=(\min\{M_1,M_2,\frac{k_2}{k_2+k_1}\})^{-1}$, 
$N_2=(\min\{M_6,M_7,M_8\})^{-1}$, $N_3=\min\{M_3, M_4\}$ and $N_4=M_5$, the system 
 \begin{equation}\label{systemwithbetahkk}
 \begin{aligned}
x_4 + C_{13}\,x_4x_5 + \beta_1\,C_{14}\,x_4x_5^2    -  T_1 &= & 0, \\
x_5  + \beta_2\,C_{23}\,x_4x_5 + \beta_2\,C_{24}\,x_4x_5^2 - T_2 &=& 0,
 \end{aligned}
 \end{equation}
has at least $3$ positive solutions. Returning to the original constants, if we keep
 fixed $k_2, k_3, k_4, k_5$,  $T_1, T_2$ and we replace $k_1, k_6$ by
$\bar k_1=(\beta_1(\frac{1}{k_1+k_2})-\frac{1}{k_2})^{-1}$ and $\bar k_6= (\beta_2)^{-1}k_6$,
 the positive steady states arising from the network with these
constants are the positive solutions of  the polynomial system (\ref{systemwithbetahkk}),
 and so it is multistationary because there are at least $3$ positive steady states in a fixed 
 stoichiometric compatibility class. Observe that $k_1$ is positive because of the choice of $N_1$.
\end{proof}

Notice that if the value of $\beta_1$ in the proof of Proposition~\ref{prop:mixedHK} is large enough, $\bar k_1$ is smaller than $k_3$, the necessary 
and sufficient condition to guarantee multistationarity that appears in \cite{feliumincheva}.

\section{Distributive multisite phosphorylation systems}\label{sec:3}
 
 Phosphorylation/dephosphorylation 
 are post-translational modification of proteins mediated by enzymes, particular proteins that add or take 
 off a phosphate group at a specific site, inducing a conformational change that allows/prevents the protein 
 to perform its function.  
 The standard building block in cell signaling is the following enzyme mechanism, which is called a Michaelis-Menten mechanism.
 \vskip -10pt
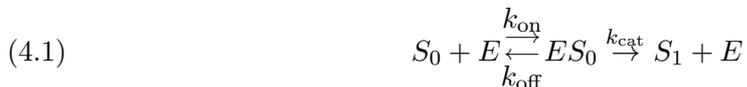
 \begin{equation}\label{eq:MM}
 {{S_0}}+{E}
 \arrowschem{k_{\rm{on}}}{k_{\rm{off}}} ES_0
 \stackrel{k_{\rm{cat}}}{\rightarrow} {{S_1}}+{E} 
 \end{equation}
  This basic network involves four species: the substrate $S_0$, the phosphorylated substrate $S_1$, the enzyme $E$, called kinase, and the intermediate species $ES_0$, and 
 $3$ reactions, with reaction constants called $k_{\rm{on}}, k_{\rm{off}}, k_{\rm{cat}}$.
 The enzyme $E$ is not consumed after the whole mechanism, which is assumed to be with mass-action kinetics. 
 The concentration of the donor of the phosphate group is considered to be constant, thus hidden in the reaction constants and ignored. 
 This mechanism with $4$ species, $3$ complexes and $3$ reactions is usually
 represented by the scheme depicted in Figure~\ref{fig:scheme}.
 \begin{figure}[h]
 \centering
{\small
\begin{tikzpicture}[node distance=0.5cm]
  \node[] at (1.1,-1.5) (dummy2) {};
  \node[left=of dummy2] (p0) {{$S_0$}};
  \node[right=of p0] (p1) {{$S_1$.}}
    edge[<-, bend right=45] node[above] {{\small{${E}$}}} (p0);
\end{tikzpicture}

    \caption{The network~\eqref{eq:MM}.}\label{fig:scheme}}
\end{figure}
 
The addition of phosphate groups to multiple sites of a single
molecule, may be distributive or processive. Distributive systems require an 
enzyme and substrate to bind several times in order to add/remove multiple
phosphate groups. Processive systems require
only one binding to add/remove all phosphate groups and it was shown in~\cite{processive} that such systems cannot admit more than
one steady state in each stoichiometric compatibility class.
The distibutive multisite phosphorylation system describes the $n$-site phosphorylation of a protein by a 
kinase/phosphatase pair in a sequential and distributive mechanism and it is
known that it has the
capacity of multistationarity for any $n \ge 2$ \cite{sontag}.
  
 The reaction mechanism for the sequential distributive mechanism for the $n$-site network is a sequence of reactions as in~\eqref{eq:MM}, where we append
 $n$ subgraphs of the form:
  \vskip -10pt
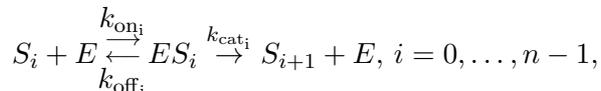
 \begin{equation*}\label{eq:MMiF}
 {{S_i}}+{E}
 \arrowschem{k_{\rm{on}_i}}{k_{\rm{off}_i}} ES_i
 \stackrel{k_{\rm{cat}_i}}{\rightarrow} {{S_{i+1}}}+{E}, \, i=0, \dots, n-1,
 \end{equation*}
 and, on the other side, $n$ subgraphs of the form:
 \begin{equation*}\label{eq:MMiF2}
 {{S_i}}+{F}
 \arrowschem{\ell_{\rm{on}_{i-1}}}{\ell_{\rm{off}_{i-1}}} FS_i
 \stackrel{\ell_{\rm{cat}_{i-1}}}{\rightarrow} {{S_{i-1}}}+{F} , \, i=1, \dots, n,
 \end{equation*}
 where $F$ denotes another enzyme called phosphatase, to obtain the network:
  
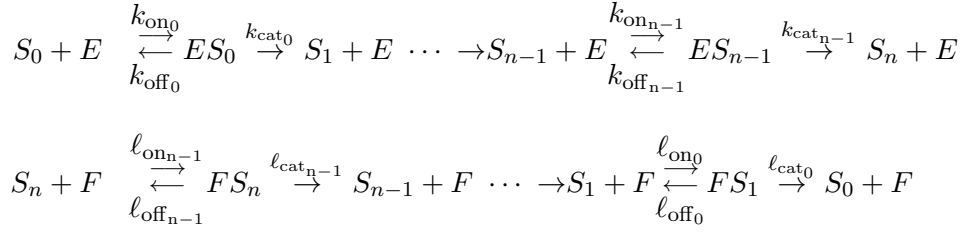
\[   
 \begin{array}{rl} 
 \nonumber 
 S_0+E &
 \arrowschem{k_{\rm{on}_0}}{k_{\rm{off}_0}} ES_0
 \stackrel{k_{\rm{cat}_0}}{\rightarrow} S_1+E\ \cdots\ {\rightarrow}S_{n-1}+E 
 \arrowschem{k_{\rm{on}_{n-1}}}{k_{\rm{off}_{n-1}}}ES_{n-1}
 \stackrel{k_{\rm{cat}_{n-1}}}{\rightarrow} S_n+ E \\
 \label{eq:net_phospho_n_complete} \\ 
 S_n+F &
 \arrowschem{\ell_{\rm{on}_{n-1}}}{\ell_{\rm{off}_{n-1}}} FS_{n}
 \stackrel{\ell_{\rm{cat}_{n-1}}}{\rightarrow} S_{n-1}+F \ \cdots\ {\rightarrow} S_{1}+F
 \arrowschem{\ell_{\rm{on}_0}}{\ell_{\rm{off}_0}} FS_1
 \stackrel{\ell_{\rm{cat}_0}}{\rightarrow} S_0+ F
 \end{array}
 \]
 It represents one substrate that can sequentially acquire up to $n$ phosphate groups, via the action of the kinase $E$, and which can be sequentially released via the action of the phosphatase $F$, in both cases via an intermediate species formed by the union of the substrate and the enzyme.
 The kinetics of this network is deduced by applying the law of mass action to this explicit labeled digraph. There are $3n + 3$ species: 
 the substrate species $S_0$, $S_1$,$\dots$,$S_n$, the enzyme species $E$ and $F$ 
 and the intermediate species $ES_0$, $ES_1$, $\dots$, $ES_{n-1}$, $FS_1$, $FS_2, \dots, FS_{n}$. We denote by $s_0$, 
 $s_1, \dots, s_n$, $e$, $f$, $y_0$, $y_1$, $\dots$, $y_{n-1}$, $u_0$, $u_1$, $\dots$, $u_{n-1}$ 
 the concentration of the species $S_0$, $S_1$,$\dots$,$S_n$, $E$, $F$, $ES_0$, $ES_1$, $\dots$, 
 $ES_{n-1}$, $FS_1$, $FS_2$, $\dots$, $FS_{n}$ respectively. The associated dynamical system that
 arises under mass-action kinetics equals:
 
 \begin{small}
 \begin{align*}
 \frac{ds_0}{dt} &=  {-k_{\rm{on}_0}}s_0e + {k_{\rm{off}_0}}y_0 + {\ell_{\rm{cat}_0}} u_0, \\
 \frac{ds_i}{dt} &=  {k_{\rm{cat}_{i-1}}}y_{i-1} - {k_{\rm{on}_{i}}}s_ie + {k_{\rm{off}_{i}}}y_{i} + {\ell_{\rm{cat}_i}}u_i -
  {\ell_{\rm{on}_{i-1}}}s_if + {\ell_{\rm{off}_{i-1}}}u_{i-1},\ i=1, \dots, n-1, \\
 \frac{ds_n}{dt}& = {k_{\rm{cat}_{n-1}}}y_{n-1}  -{\ell_{\rm{on}_{n-1}}}s_n f + {\ell_{\rm{off}_{n-1}}} u_{n-1},\\
 \frac{dy_i}{dt} &=  {k_{\rm{on}_i}}s_i e -{(k_{\rm{off}_i}+k_{\rm{cat}_i})}y_i, \ i=0,\dots,n-1, \\
  \frac{du_i}{dt} &=  {\ell_{\rm{on}_i}}s_{i+1} f -{(\ell_{\rm{off}_i}+\ell_{\rm{cat}_i})}u_i, \ i=0,\dots,n-1, \\
 \frac{de}{dt}&= -\sum_{i=0}^{n-1}\frac{dy_i}{dt},\ \ \frac{df}{dt}= -\sum_{i=0}^{n-1}\frac{du_i}{dt}.
 \end{align*}
 \end{small}
 
There are three linearly independent conservation laws for any value of $n$ (and no more):
 \begin{small}
 \begin{equation}\label{eq:consfosfo}
 \sum_{i=0}^n s_i + \sum_{i=0}^{n-1} y_i + \sum_{i=0}^{n-1} u_i=  S_{tot}, \quad
 e + \sum_{i=0}^{n-1} y_i=  E_{tot}, \quad
 f + \sum_{i=0}^{n-1} u_i=  F_{tot},
\end{equation}
 \end{small}
 \kern-.4em where clearly the total amounts $S_{tot}, E_{tot}, F_{tot} \in \R_{>0}$ for
 any trajectory intersecting the positive orthant.
 It is straigthforward to see from the differential equations that the concentrations of the intermediates
  species at steady state satisfy the following binomial equations:
 \begin{equation}\label{eq:intermediates}
 y_i \, - \, K_i \, es_i=0,\, i=0,\dots,n-1, \qquad u_i \, - \, L_i \, fs_{i+1}=0,\,  i=0,\dots,n-1,\end{equation}
 where $K_i=\frac{k_{\rm{on}_i}}{k_{\rm{off}_i}+k_{\rm{cat}_i}}$ and  $L_i=\frac{\ell_{\rm{on}_i}}{\ell_{\rm{off}_i}+\ell_{\rm{cat}_i}}$ 
 for each $i=0,\dots,n-1$ ($K_i^{-1}$ and $L_i^{-1}$ are usually called \textit{Michaelis-Menten constants}, $i=0,\dots,n-1$).
 
 Sequential phosphorylation mechanisms are an example of $s$-toric MESSI networks, defined in~\cite{aliciaMer}. We recall in Section~\ref{sec:4} their
 definition and we present  general results  for $s$-toric  MESSI systems that explain our computations in this section.
 In particular, by Theorem~3.5 in \cite{aliciaMer} we can find the following binomial equations that describe the steady states.
 The whole steady state variety can be cut out in the positive orthant by adding to the binomials in~\eqref{eq:intermediates},
  the binomial equations:
 \[\tau_is_ie - \nu_is_{i+1}f = 0,\]
 where $\tau_i=k_{\rm{cat}_i}K_i$ and $\nu_i=\ell_{\rm{cat}_i}L_i$, for each $i=0,\dots, n-1$.  Using these binomial equations, 
 we can parametrize the positive steady states by monomials. For instance, 
 we can write the concentration at steady state of all species in terms of the species $s_0, e, f$:
\[ \begin{array}{lcrl}
 \label{parametrizationphospo}
 s_i &=&T_{i-1} \, \frac{s_0e^i}{f^i},& i=1,\dots,n, \nonumber \\
 y_i &=& K_i \, T_{i-1} \, \frac{s_0e^{i+1}}{f^i},& i=0,\dots,n-1, \\
 u_i &=& L_i \, T_i \, \frac{s_0e^{i+1}}{f^i}, &  i=0,\dots,n-1, \nonumber
 \end{array}\]
 where $T_i=\prod^i_{j=0}\frac{\tau_j}{\nu_j}$ for $i=0,\dots,n-1$ and $T_{-1}=1$.

We will use this parametrization in order to apply Theorems~\ref{th:BS} and~\ref{thm:main2}  to the sequential phosphorylation 
mechanisms for any
$n$: 

 \begin{thm}\label{th:nfosfo}
 With the previous notations, assume
 \begin{equation}\label{cond1fosfo} 
 S_{tot} >F_{tot}.
 \end{equation}
 Then, there is a choice of rate constants for which the distributive $n$-site phosphorylation system is multistationary. 
 More explicitly,  for any choice of positive real numbers $k_{\rm{cat}_1},\ell_{\rm{cat}_1}$ satisfying
 \begin{equation}\label{cond2fosfo} 
 \frac{k_{\rm{cat}_1}}{\ell_{\rm{cat}_1}} \, > \, \max\left\{ \frac{F_{tot}}{S_{tot}-F_{tot}},\frac{F_{tot}}{E_{tot}} \right\},
 \end{equation}
 fix any value of the remaining rate constants and positive numbers $h_i$, for $i=4,\dots,2n+3$ such that $i \, 
h_{n+5}<h_{i+3}$ 
 for $i=1,\dots,n$ and $(i-1)\, h_{n+5}<h_{n+i+3}$ for $i=1,3,\dots,n$. 
 Then, there exists $t_0 >0$ such that for any value of $t \in (0, t_0)$ the system is multistationary after the rescalings
 $t^{h_{n+4}}\, k_{\rm{on}_0}$, $t^{h_{n+4+i}-h_{i+3}} \, k_{\rm{on}_i}, i=1,\dots,n-1$, $t^{h_{n+4+i}-h_{i+4}} \, \ell_{\rm{on}_i}, i=0,\dots,n-1$.
 
Similarly, for any fixed choice of reaction rate constants and total conservation constants satysfying (\ref{cond1fosfo}) and (\ref{cond2fosfo}), 
there exist positive constants $M_i$, $i=1,\dots,4n-2$ such that for any positive values of $\gamma_i$, $i=1,\dots,2n$ verifying
 \begin{align}\label{eq:gammaR}
 \gamma_i&<M_i, i=1,\dots,2n, \qquad \frac{\gamma_i}{\gamma^{i}_{n+2}} <
 M_{2n+i}, i=1,\dots,n, \\ \frac{\gamma_{n+i}}{\gamma^{i-1}_{n+2}} &<M_{3n-2+i},  i=3,\dots,n, \nonumber
 \end{align}
 the rescaling of the given parameters
  $k_{on_1}$, $k_{on_i}$, $i=2,\dots,n-1$, $\ell_{on_i}$, $i=1,\dots,n-1$ by 
   \begin{equation}\label{eq:kR}
    \gamma_{n+1} \, k_{on_1},  \, \frac{\gamma_
 {n+1+i}}{\gamma_i} \, k_{on_i}, i=2,\dots,n-1,  \, \frac{\gamma_{n+1+i}}{\gamma_{i+1}} \, \ell_{on_i}, i=1,\dots,n-1,
 \end{equation}
 respectively, gives raise to a multistationary system.
 \end{thm}
  
 \begin{proof}  Previously in this section, we showed that we can write the concentration 
 at steady state of all species in terms of the species $(s_0, e, f)$, 
 as in (\ref{parametrizationphospo}).  We substitute this monomial parametrization of 
 the steady states into the linear conservation relations \eqref{eq:consfosfo}. 
 We have a system of three equations and we write it in  matricial form:
 \[C\begin{pmatrix}
 s_0 & e & f & s_0ef^{-1} & \dots & s_0e^nf^{-n} & s_0ef^0 & \dots & s_0e^nf^{-(n-1)} & 1
 \end{pmatrix}^t=0,
 \]
 where the matrix $C\in \R^{3\times (2n+4)}$ is the matrix of coefficients:

 \begin{small}
 \[C=\begin{pmatrix}
 1 & 0 & 0 & T_0 & \dots & T_{n-1} & K_0 + L_0T_0 & \dots & K_{n-1}T_{n-2} + L_{n-1}T_{n-1} & -S_{tot} \\
 0 & 1 & 0 & 0 & \dots & 0 & K_0 & \dots & K_{n-1}T_{n-2} & -E_{tot}\\
 0 & 0 & 1 & 0 & \dots & 0 & L_0T_0 &  \dots & L_{n-1}T_{n-1} & -F_{tot}
 \end{pmatrix}.
 \]
 \end{small}

 If we order the variables in this way: $s_0$, $e$, $f$, the support of the system is:
 \begin{align*}
 \mathcal{A}=\{(1,0,0), (0,1,0), (0,0,1), (1,1,-1),(1,2,-2),\dots,(1,n,-n),\\
 (1,1,0),(1,2,-1),\dots, (1,n,-(n-1)),(0,0,0) \}. 
 \end{align*} 
 
 We want to find two positively decorated $3$-simplices with vertices in $\mathcal{A}$ which share a facet. For example we take the simplices
\begin{align*}
\Delta_1&=\{(1,0,0),(0,1,0),(0,0,1),(0,0,0)\},\\
\Delta_2&=\{(1,0,0),(0,1,0),(1,2,-1),(0,0,0)\}.
\end{align*} 
They are shown in  Figure~\ref{fig:2simplicesfosfo}, made with {\tt Polymake}~\cite{polymake}, which is a very useful 
tool to visualize and to do computations with polytopes and triangulations.
 
  \begin{figure}[h]
 \centering
 \begin{tikzpicture}%
 [x={(0.545871cm, -0.299905cm)},
 y={(0.837775cm, 0.181354cm)},
 z={(0.012573cm, 0.936572cm)},
 scale=2.000000,
 back/.style={dotted, thin},
 edge/.style={color=black, thick},
 facet/.style={fill=white,fill opacity=0.400000},
 vertex/.style={inner sep=1pt,circle,draw=black,fill=black,thick,anchor=base}]
 \coordinate (0.00000, 0.00000, 0.00000) at (0.00000, 0.00000, 0.00000);
 \coordinate (0.00000, 0.00000, 1.00000) at (0.00000, 0.00000, 1.00000);
 \coordinate (0.00000, 1.00000, 0.00000) at (0.00000, 1.00000, 0.00000);
 \coordinate (1.00000, 0.00000, 0.00000) at (1.00000, 0.00000, 0.00000);
 \coordinate (1.00000, 2.00000, -1.00000) at (1.00000, 2.00000, -1.00000);
 \draw[edge] (0.00000, 0.00000, 0.00000) -- (0.00000, 1.00000, 0.00000);
 \draw[edge] (0.00000, 0.00000, 0.00000) -- (1.00000, 2.00000, -1.00000);
 \draw[edge] (0.00000, 0.00000, 1.00000) -- (0.00000, 1.00000, 0.00000);
 \draw[edge] (0.00000, 1.00000, 0.00000) -- (1.00000, 2.00000, -1.00000);
 \node[vertex] at (0.00000, 1.00000, 0.00000)     {};
 \fill[facet, fill=red!70!black,fill opacity=0.800000] (1.00000, 0.00000, 0.00000) -- (0.00000, 0.00000, 0.00000) -- (0.00000, 0.00000, 1.00000) -- cycle {};
 \fill[facet, fill=green!70!black,fill opacity=0.800000] (1.00000, 2.00000, -1.00000) -- (0.00000, 1.00000, 0.00000) -- (1.00000, 0.00000, 0.00000) -- cycle {};
 \fill[facet, fill=red!70!black,fill opacity=0.800000] (1.00000, 0.00000, 0.00000) -- (0.00000, 1.00000, 0.00000) -- (0.00000, 0.00000, 1.00000) -- cycle {};
 \draw[edge] (0.00000, 0.00000, 0.00000) -- (0.00000, 0.00000, 1.00000);
 \draw[edge] (0.00000, 0.00000, 0.00000) -- (1.00000, 0.00000, 0.00000);
 \draw[edge] (0.00000, 0.00000, 1.00000) -- (1.00000, 0.00000, 0.00000);
 \draw[edge] (1.00000, 0.00000, 0.00000) -- (1.00000, 2.00000, -1.00000);
 \draw[edge] (1.00000, 0.00000, 0.00000) -- (0.00000, 1.00000, 0.00000);
 \node[vertex] at (0.00000, 0.00000, 0.00000)     {};
 \node[vertex] at (0.00000, 0.00000, 1.00000)     {};
 \node[vertex] at (1.00000, 0.00000, 0.00000)     {};
 \node[vertex] at (1.00000, 2.00000, -1.00000)    {};
 \node at (-0.65000, 0.00000, -0.20000){$(0,0,0)$};
  \node at (-0.65000, 0.00000, 0.90000){$(0,0,1)$};
\node at (1.0000, 0.00000, -0.2500){$(1,0,0)$};
 \node at (0.3500, 1.20000, 0.2000){$(0,1,0)$};
 \node at (1.5000, 2.0000, -1.0000){$(1,2,-1)$};
 \end{tikzpicture}
 \caption{The simplices $\Delta_1$ and $\Delta_2$.}
 \label{fig:2simplicesfosfo}
 \end{figure}
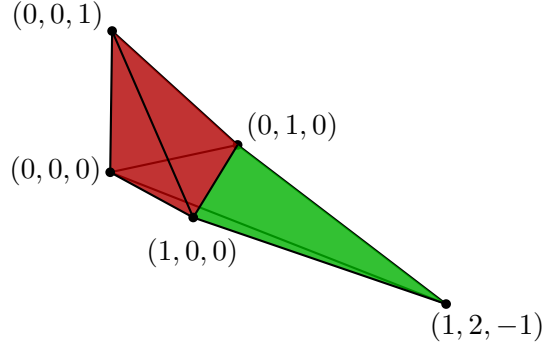

 The simplex $\Delta_1$ is automatically positively decorated by $C$. The simplex $\Delta_2$  is positively decorated by $C$ if and only if:
 
 $$E_{tot} - \dfrac{K_1T_0F_{tot}}{{L_1T_1}} >0 , \quad \text{and }  S_{tot} - \dfrac{(K_1T_0 + L_1T_1)F_{tot}}{L_1T_1}>0.$$
 
 Getting back to the original constants, we can write the previous conditions in the following form:
 %
 %
 \begin{equation} \label{eq:conditionsf2v2}
 \begin{aligned}
 S_{tot} &>F_{tot},\\
 \frac{k_{\rm{cat}_1}}{\ell_{\rm{cat}_1}} \, > \, \max&\left\{ \frac{F_{tot}}{S_{tot}-F_{tot}},\frac{F_{tot}}{E_{tot}} \right\}.
 \end{aligned}
 \end{equation}
 
 Suppose that conditions (\ref{eq:conditionsf2v2}) hold. Then the simplices $\Delta_1$ and $\Delta_2$ are positively decorated. 
Proposition~\ref{prop:2} says that exists a regular triangulation of $\Gamma$ of the convex hull of $\mathcal{A}$, 
 such that the two simplices $\Delta_1$ and $\Delta_2$ are part of that triangulation. Given any height function $h$
 inducing such a $\Gamma$, by Theorem~\ref{th:BS} there exists $t_0 \in \R_{>0}$ such that for all 
 $0<t<t_0$, the number of positive nondegenerate solutions of the scaled system:
 \begin{small}
 \begin{equation}\label{systemwitht-n}
 \begin{aligned}
 t^{h_1}s_0 + \sum_{i=1}^n T_{i-1}t^{h_{i+3}}\dfrac{s_0e^{i}}{f^{i}} + 
 \sum_{i=0}^{n-1}(K_iT_{i-1}+L_iT_i)t^{h_{n+4+i}}\dfrac{s_0e^{i+1}}{f^{i}} 
  - S_{tot}t^{h_{2n+4}}&= 0,\\
 t^{h_2}e  + \sum_{i=0}^{n-1}K_i T_{i-1}t^{h_{n+4+i}}\dfrac{s_0e^{i+1}}{f^{i}}  - E_{tot}t^{h_{2n+4}}&= 0, \\
 t^{h_3}f + \sum_{i=0}^{n-1}L_i T_it^{h_{n+4+i}}\dfrac{s_0e^{i+1}}{f^{i}} - F_{tot}t^{h_{2n+4}}&= 0,
 \end{aligned}
 \end{equation}
 \end{small}
 \noindent is at least two, where  $h_1=h(1,0,0)$, $h_2=h(0,1,0)$, $h_3=h(0,0,1)$, $h_{i+3}=h(1,i,-i)$,
 for $i=1,\dots,n$, $h_{n+3+i}=h(1,i,-(i-1))$, for $i=1,\dots,n$ and $h_{2n+4}=h(0,0,0)$.
 
 We can suppose without loss of generality that  $h_1=h_2=h_3=h_{2n+4}=0$ and $h(1,2,-1)=h_{n+5}>0$. Let $\varphi_1$ 
  and $\varphi_2$ be the affine linear functions $\varphi_1(x,y,z)=0$ and $\varphi_2(x,y,z)=-h_{n+5} \, z$ which agree with $h$ 
on the 
  simplices $\Delta_1$ and $\Delta_2$ respectively.
   Then,
   \vskip -10pt
   \[\begin{matrix}
   0 &<& h_{i+3}, &\varphi_2(1,i,-i)  & = & h_{n+5}\,i & < & h_{i+3}& \mbox{ for }i=1,\dots,n,\\
   0 &<& h_{n+3+i}, & \varphi_2(1,i,-(i-1))  & = & h_{n+5}(i-1) & < & h_{n+3+i}& \mbox{ for }i=1,3\dots,n,\ i\neq2.
   \end{matrix}\]
  Any such choice defines a regular subdivision containing both simplices (and if the heights are generic the subdivision
  is a regular triangulation).

 If we rescale the following constants:
 \begin{align} \label{eq:resc}
 &t^{h_{n+4}}\, K_0,\qquad t^{h_{n+4+i}-h_{i+3}} \, K_{i},\  i=1,\dots,n-1\\ \nonumber
 &t^{h_{n+4+i}-h_{i+4}}\, L_i,\  i=0, \dots, n-1.
 \end{align}
 and we keep fixed the values of the constants $k_{\rm{cat}_1}$ and $\ell_{\rm{cat}_1}$ and the total values $E_{tot}$, $F_{tot}$ 
 and $S_{tot}$ (such that (\ref{eq:conditionsf2v2}) holds), the dynamical system obtained from the network with these constants is the 
 system (\ref{systemwitht-n}). And then, for these constants the network has at least two positive steady states.  
 Moreover, it is straightforward to check that it is enough to rescale the following original
 constants as indicated in the statement:
 \begin{equation}\label{eq:rescalek}
 t^{h_{n+4}}\, k_{\rm{on}_0}, \,  t^{h_{n+4+i}-h_{i+3}}\, k_{\rm{on}_i}, i=1,\dots,n-1, \, t^{h_{n+4+i}-h_{i+4}}\, \ell_{\rm{on}_i}, i=0,\dots,n-1,
 \end{equation}
  to get the equalities~\eqref{eq:resc}.
 
The last part of the statement follows with similar arguments via Proposition~\ref{prop:cone} and Theorem~\ref{thm:main2}.
 
 \end{proof}

\begin{remark}  Using a parametrization of the concentrations of the species
 at steady state in terms of other variables (or with another choice of the simplices) 
 we can obtain other regions in the parameters space that guarantee multistationarity .
\end{remark}

 
 \section{MESSI Systems}\label{sec:4}
 
 In \cite{aliciaMer}, Dickenstein and P\'erez Mill\'an introduced a general framework for biological systems, called MESSI systems, 
 that describe Modifications of type Enzyme-Substrate or Swap with Intermediates. Distributive multisite phosphorylation systems and 
 enzymatic cascades as the one we depict in Figure~\ref{fig:sadi}, with any number of layers which occur in cell signaling pathways (and we study in detail in~\cite{cascades}), are examples of MESSI systems of biological significance. In particular they are examples of $s$-toric MESSI systems, 
 an important subclass of MESSI systems.
 The authors proved in~\cite{aliciaMer} that any $s$-toric MESSI system is toric, that is, the positive steady states can 
 be described with binomials and under certain hypotheses, 
 they can choose explicit binomials with coefficients in $\Q(\kappa)$ which describe the positive steady states. Moreover, under certain combinatorial conditions,
 they describe a basis of conservation laws for these systems. 
 
 \begin{figure}[h]
	\centering
		 \begin{tikzpicture}[scale=0.7,node distance=0.5cm]
		\node[] at (1.1,-1.5) (dummy2) {};
		\node[left=of dummy2] (p0) {$P_0$};
		\node[right=of p0] (p1) {$P_1$}
		edge[->, bend left=45] node[below] {\textcolor{black!70}{\small{$F_2$}}} (p0)
		edge[<-, bend right=45] node[above] {} (p0);
		\node[] at (-1,0) (dummy) {};
		\node[left=of dummy] (s0) {$S_0$};
		\node[right=of s0] (s1) {\textcolor{blue}{$S_1$}}
		edge[->, bend left=45] node[below] {\textcolor{black!70}{\small{$F_1$}}} (s0)
		edge[<-, bend right=45] node[above] {\textcolor{black!70}{\small{$E$}}} (s0)
		edge[->,thick,color=blue, bend left=25] node[above] {} ($(p0.north)+(25pt,10pt)$);
		\end{tikzpicture}
	
		\caption{Scheme of a 2-layer cascade of GK-loops,  similar to Fig.~\ref{fig:scheme}.}\label{fig:sadi}
\end{figure}
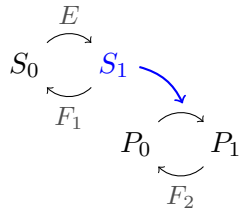

In order to apply our method in Section~\ref{sec:1} to determine a multistationarity
region for the network~\eqref{networkHHK} or in Section~\ref{sec:3} 
for the distributive multisite phosphorylation systems, we proposed and replaced a rational
parametrization of the steady state variety into a basis of the conservation
relations, and we then had to rescale some of original parameters at the end of the procedure (as in~\eqref{eq:rescalek}, ~\eqref{eq:kR}).
Our main result in this section is Theorem~\ref{th:reescalamientos}, which  guarantees that the rescaling of the parameters can be done
for any $s$-toric MESSI system, together with Proposition~\ref{prop:parametrizationmessi} which
ensures and describes the existence of a rational parametrization of the
steady state variety. 

\subsection{Basic definitions concerning MESSI systems}
We briefly introduce the basic definitions of MESSI systems. For a more detailed explanation, see \cite{aliciaMer}.

A MESSI network is a chemical reaction network, for which exists a partition of 
the set of species $\Sp$ into disjoint subsets:
 \begin{equation}\label{eq:partition}
\Sp=\Sp^{(0)}\bigsqcup \Sp^{(1)} \bigsqcup \Sp^{(2)} \bigsqcup \dots \bigsqcup \Sp^{(m)},
\end{equation}
where $m \ge 1$ and $\bigsqcup$ denotes disjoint union. Species in $\Sp^{(0)}$ are called \emph{intermediate} 
and species in $\Sp_1:= \Sp \setminus\Sp^{(0)}$ are called \emph{core}, with $\# \Sp^{(0)}=p$ and $\# \Sp^{(1)}=s-p>0$.
As before, we denote the species with upper letters and the concentration of the species with small letters,
 for example $x_j$ denotes the concentration of the species $X_j$.


There are two types of complexes allowed in a MESSI network: intermediate complexes and core complexes. 
The intermediate complexes are complexes that
consist of a unique intermediate species that only appears in that complex. 
The core complexes are mono or bimolecular and consist of either one or two core species.
When a core complex consists of two species $X_i, X_j$, they { must} 
belong to {different} sets $\Sp^{(\alpha)}, \Sp^{(\beta)}$ with $\alpha \neq \beta,\, \alpha, \beta \geq1$. 

We say that a complex $y$  reacts to a complex $y'$ \emph{via intermediates} if either $y\to y'$ or
there exists a path of reactions from $y$ to $y'$ only through intermediate complexes.  
This is denoted  by $y \uri y'$. Another condition in the intermediate complexes 
is that for every intermediate complex $y$, there must exist core
complexes $y_1$ and $y_2$ such that $y_{1}\uri y$ and $y\uri y_2$.
The reactions in a MESSI network satisfy the following rules:
if three species are related by $X_i+ X_j \uri X_k$ or $X_k \uri X_i + X_j$, then $X_k$ is an intermediate species.
If two monomolecular complexes consisting of a single core species $X_i, X_j$ are related by $X_i\uri X_j$,
 then there exists $\alpha \ge 1$ such that both belong to $\Sp^{(\alpha)}$.
And if $X_i+X_j\uri X_k+X_\ell$ then, there exist $\alpha \neq \beta$ such that 
$X_i,X_k \in \Sp^{(\alpha)}$, $X_j,X_\ell \in \Sp^{(\beta)}$ or $X_i,X_\ell \in \Sp^{(\alpha)}$, $X_j,X_k \in \Sp^{(\beta)}$. 

A partition in the set of species that satisfies all the previous conditions in the complexes 
and reactions defines a MESSI \emph{structure}. There can be many possible partitions that define a MESSI structure of a fixed network.
If we have two partitions
$\Sp = \Sp^{(0)}\sqcup \Sp^{(1)} \sqcup \Sp^{(2)} \sqcup \dots \bigsqcup \Sp^{(m)}$ 
and $\Sp = {\Sp'}^{(0)}\sqcup \Sp'^{(1)} \sqcup \Sp'^{(2)} \sqcup \dots \bigsqcup \Sp'^{(m')}$, 
we say that the first partition refines the second one if and only if $\Sp^{(0)}\supseteq {\Sp'}^{(0)}$
and for any $\alpha \ge 1$, there exists
$\alpha'\ge 1$ such that $\Sp^{(\alpha)}\subseteq {\Sp'}^{(\alpha')}$.
This defines a partial order in the set of all possible partitions, and in particular we have the notion of a \emph{minimal} partition. 

\begin{ej}\label{ej:mixedphospho}  In Section~\ref{sec:3} 
we presented the distributive multisite phosphorylation systems. The following network is an example of 
a mixed phosphorylation mechanism (partially distributive, 
partially processive) studied in~\cite{suwanmajokrishnan}. The reaction network is as follows:
  
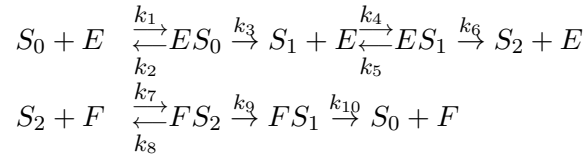
\[   
 \begin{array}{rl} 
 S_0+E &
 \arrowschem{\mbox{\footnotesize $ k_1$ }}{\mbox{\footnotesize $ k_2$ }} ES_0
 \stackrel{k_3}{\rightarrow} S_1+E
 \arrowschem{\mbox{\footnotesize $ k_4$ }}{\mbox{\footnotesize $ k_5$ }} ES_1
 \stackrel{k_6}{\rightarrow} S_2+E
 \label{eq:net_phospho_mix} \\ 
 S_2+F &
 \arrowschem{\mbox{\footnotesize $ k_7$ }}{\mbox{\footnotesize $ k_8$ }} FS_2
 \stackrel{k_9}{\rightarrow} FS_1
\stackrel{k_{10}}{\rightarrow} S_0+ F
 \end{array}
 \]
A MESSI structure of the network is given by the following minimal partition of the species:
 $\Sp^{(0)}=\{ES_0, ES_1, FS_1, FS_2\}$  (the intermediate species), $\Sp^{(1)}=\{E\}$, $\Sp^{(2)}=\{F\}$ and $\Sp^{(3)}=\{S_0, 
S_1, S_2\}$.
We will use this partition in the next examples featuring this network.
Another  example of a partition giving a MESSI structure, which is not minimal, is the following:
 ${\Sp'}^{(0)}=\{ES_0, ES_1, FS_1, FS_2\}$, ${\Sp'}^{(1)}=\{E,F\}$ and ${\Sp'}^{(2)}=\{S_0, 
S_1, S_2\}$. 
\end{ej}

We now present three digraphs associated to a MESSI network with digraph $G$.
First, we introduce the associated digraph $G_1$, where the intermediate species are eliminated, 
that is, with set of species $\Sp_1$. We associated to this set the inherited partition
\begin{equation}\label{eq:G1}
 \Sp_1= \Sp^{(1)} \bigsqcup \Sp^{(2)} \bigsqcup \dots \bigsqcup \Sp^{(m)}.
\end{equation}
The vertex set of $G_1$ consists of all the core complexes. 
An edge $y\to y'$, with $y, y'$ core complexes, belongs to the edge set of $G_1$ if and only if $y\uri y'$ in $G$. 
The explicit rate constants of these edges are described in the proof of Theorem~3 in the ESM of \cite{fw13} (also made explicit 
in~\cite{aliciaMer}).
We will use the following notation for the intermediate species: 
$\Sp^{(0)}=\{U_1,\dots,U_p\}$.
For each $y\uri y'$ in $G$, with $y$ and $y'$ core complexes,   
the reaction constant $\tau(\kappa)$ in $G_1$ which gives the label $y\overset{\tau(\kappa)}{\longrightarrow} y'$, equals:
\begin{equation}\label{eq:tau}
\tau(\kappa)=\kappa_{yy'}+\overset{p}{\underset{k=1}{\sum}}\kappa_k\, \mu_k(\kappa),
\end{equation}
where $\kappa_{yy'}\ge 0$ is positive when $y\overset{\kappa_{yy'}}{\longrightarrow} y'$ in $G$ (and $\kappa_{yy'}=0$ otherwise), 
and $\kappa_k\ge 0$ is positive if $U_k\overset{\kappa_k}{\longrightarrow} y'$ and 
$y\uri U_k$ in $G$ (and $\kappa_k=0$ otherwise), with $\mu_k(\kappa)$ as in \eqref{eq:monom} below.

We next introduce a labeled associated multidigraph $G_2$ where we ``hide'' the concentrations of some of the species in the labels. 
We keep all monomolecular reactions $X_i\to X_j$ in $G_1$ and for each reaction $X_i+X_\ell \overset{\tau}{\longrightarrow} X_j+X_m$ in $G_1$,
with $X_i,X_j \in \Sp^{(\alpha)}$, $X_\ell,X_m \in \Sp^{(\beta)}$,
we consider two reactions $X_i \overset{\tau x_\ell}{\longrightarrow} X_j$
and $X_\ell \overset{\tau x_i}{\longrightarrow} X_m$. In principle this multidigraph $MG_2$ might contain loops or parallel edges between any pair of nodes. 
We obtain the  digraph $G_2$ by collapsing into one edge all parallel edges of $MG_2$.
The label of an edge in $G_2$ is the sum of the labels of the parallel edges in
the multidigraph.  By the rules of the reactions in a MESSI network, $G_2$ is a linear
graph (each node is indicated by a single variable) and the labels on the edges depend on the rate constants but might also depend on the concentrations of some species.  We call $G_2^\circ$ the 
digraph obtained from the deletion of loops and isolated nodes of $G_2$.
It can be shown (see Lemma~18 of \cite{aliciaMer}) that if the partition associated to a MESSI system is minimal, the connected 
components of the associated digraph $G_2$ are in bijection with the subsets $\Sp^{(\alpha)}$ corresponding to a core species and the set of nodes of the corresponding component equals $\Sp^{(\alpha)}$.

Finally, given a MESSI system with a minimal partition of the set of species, we define the associated digraph $G_E$, whose vertices are the 
sets $\Sp^{(\alpha)}$ for $\alpha \ge 1$, and there is an edge from $\Sp^{(\alpha)}$ to $\Sp^{(\beta)}$  if there is a species in $\Sp^{(\alpha)}$
in a label of an edge in $G_2^\circ$ between species of $\Sp^{(\beta)}$.

\begin{ej}[Example~\ref{ej:mixedphospho}, continued] The digraphs  $G_1$, $G_2$, and $G_E$ 
associated to the network of Example~\ref{ej:mixedphospho} are depicted in Figure~\ref{fig:ejdigraphsmixedphospho}.

 \begin{figure}
 \centering
 \begin{tikzpicture}
 \node[]at(0,0)(a){$\begin{array}{c}
  S_0 + E  \overset{\tau_1}{\rightarrow} S_1 + E \overset{\tau_2}{\rightarrow}
S_2 + E\\
 S_2 + F  \overset{\tau_3}{\rightarrow} S_0 + F 
  \end{array}$};
  \node[]at(0,-1.5){$G_1$};
\end{tikzpicture} 
 \hspace{1.0cm}
\begin{tikzpicture}
\node[]at(-1,0.8)(a){$S_0$};
\node[]at(0,0.8)(b){$S_1$};
\node[]at(1,0.8)(c){$S_2$};
\node[]at(-1,-0.4)(e){$E$};
\node[]at(0.5,-0.4)(f){$F$};
\path[->,every node/.style={font=\sffamily\footnotesize}]
    (a) edge node [above] {$e\tau_1$} (b)
    (b) edge node [above] {$e\tau_2$} (c)
    (c) edge[bend left] node [below] {$f\tau_3$} (a);
\path
(e)   edge[in=0,out=300,loop] node  {} (e);
\path
(f)   edge[in=0,out=300,loop]node  {} (f);
\node[]at(0,-1.5){$G_2$};
\end{tikzpicture} 
 \hspace{1.0cm}
\begin{tikzpicture}[ampersand replacement=\&] 
\matrix (m) [matrix of math nodes, row sep=1.5em, column sep=1em, text height=1.5ex, text depth=0.25ex]
{ \Sp^{(1)}\&  \Sp^{(3)}  \\
	\Sp^{(2)} \&  \& \\};
\draw[->]($(m-1-1)+(0.4,0)$) to node[below] (x) {} ($(m-1-2)+(-0.4,0)$);
\draw[->]($(m-2-1)+(0.4,0)$) to node[below] (x) {} ($(m-1-2)+(-0.35,-0.15)$);
\node[]at(0,-1.5){$G_E$};
\end{tikzpicture}
 \caption{The digraphs $G_1$, $G_2$ and $G_E$ of the network in Example~\ref{ej:mixedphospho}.}
\label{fig:ejdigraphsmixedphospho}
 \end{figure}
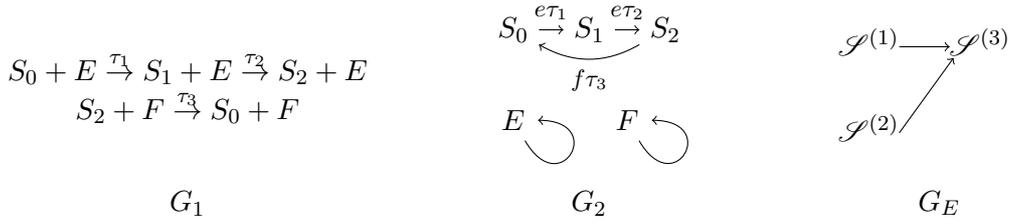
\end{ej}

In order to give the definition of an $s$-toric MESSI we have to recall some definitions from graph theory. Recall that a vertex
in a directed graph has indegree zero if it is not the head of any directed edge and outdegree zero if it is not the tail of any directed edge.
A spanning tree of a digraph is a subgraph that contains all the vertices, is connected and acyclic as an undirected graph. 
An $i$-tree of a graph is a spanning tree where the $i$-th vertex
is its unique sink (that is, the only vertex with outdegree zero). Given an $i$-tree $T$, we call $c^{T}$ the product of the labels of all the edges of $T$.
An $s$-toric MESSI system is a MESSI system that also satisfies the following conditions: i) for any intermediate complex $y$, there exists a 
unique core complex $y_1$ such that $y_1\uri y$, ii) the associated multidigraph $MG_2$ does not have parallel edges and the digraph $G_2$ is
weakly reversible (i.e., for any pair of nodes in the same connected component there is a directed path joining them), iii) for each vertex $i$ of $G_2^\circ$ and any choice of $i$-trees $T, T'$ of $G_2^\circ$, 
the quotient $c^T/c^{T'}$ only depends on the rate constants $\tau$.
It is interesting to note that even if this definition is restrictive, many of the
common enzymatic networks in the literature satisfy these conditions. So, there is
a wide applicability of our results.



\subsection{Existence of rescalings}
The following proposition summarizes some results of \cite{aliciaMer} and describes the conservation laws as well as
the existence of a positive parametrization of the positive steady states. By a positive parametrization of the variety $V$
of positive steady states we mean a $\mathcal{C}^1$ and bijective function
$$\phi\colon \R^m_{>0}\to V\cap \R^s_{>0},$$
$$\bar x=(\bar x_1, \dots, \bar x_m) \mapsto (\phi_1 (\bar x),\dots, \phi_n(\bar x)),$$ for some $m<s$. 
Proposition~\ref{prop:parametrizationmessi} also features the form of the system when we replace the concentration at steady state 
of the species by its parametrization into the conservation laws, which is our procedure when we apply our results to the question of determining regions
of multistationarity of biochemical reaction networks. 

\begin{prop}\label{prop:parametrizationmessi} Let $G$ be the underlying digraph of a MESSI system with  
fixed reaction rate constants $\kappa$. Consider a minimal partition of the set of species as in \eqref{eq:partition}
and the associated digraphs $G_2$ and $G_E$ defined above. 
Suppose that the system is $s$-toric, $G_E$ has no directed cycles and assume that any pair of nodes in the same connected component 
of $G_2$ are connected by a unique simple path.\footnote{A simple path is a path that visits each vertex exactly once.}

Choose $m$ species $X_{i_1},\dots,X_{i_m}$, such that $X_{i_{\alpha}}\in\Sp^{(\alpha)}$ for $\alpha=1,\dots,m$.
Then, 
there exists an explicit basis of $m$ conservation laws with coefficients $0,1$
and a positive monomial parametrization of the concentrations of the species 
at steady state in terms of the $m$ concentration variables $x_{i_1},\dots,x_{i_m}$. Moreover, if we replace 
the concentrations of the species by its parametrization in these conservations laws we obtain a system of the form:
\begin{equation}\label{conslawparam}
 \ell_{\alpha}(x,\kappa):= \sum_{j=1}^n \varphi_{\alpha,j}(\kappa)x^{a_j}=T_{\alpha}, \quad \alpha=1,\dots,m,
 \end{equation}
where $x=(x_{i_1},\dots,x_{i_m})$, with $a_j\in\Z^m- \{0\}$ for each $j=1,\dots, n$, 
for some constants $T_{\alpha}$ which are positive if the trajectory intersects the positive orthant for each $\alpha=1,\dots,m$. Here
$\varphi_{\alpha,j}(\kappa)$ is a positive rational function in the reaction rate constants for each $\alpha=1,\dots,m$ and $j=1,\dots,n$.
\end{prop}

Now, we state the main result of this section, which guarantees that the rescaling of the original parameters $\kappa$ can always be done in our setting. Its proof
can be implemented as an  algorithm.

\begin{thm}\label{th:reescalamientos} Let $G$ be the underlying digraph of a MESSI system. Consider a minimal partition of the set of species, 
and the associated digraphs $G_2$ and $G_E$ defined as before. Suppose that the system is $s$-toric, $G_E$ has no directed cycles and assume 
that any pair of nodes in the same connected component of $G_2$ are connected by a unique simple path. 
Fix $m$ species $X_{i_1},\dots,X_{i_m}$, such that $X_{i_{\alpha}}\in\Sp^{(\alpha)}$ for $\alpha=1,\dots,m$ 
and consider the parametrization and the system~\eqref{conslawparam} obtained in Proposition~\ref{prop:parametrizationmessi}.

Given $\gamma\in\R^{n+1}_{>0}$, 
 reaction rate constants $\kappa$ and total conservation constants $T_{\alpha}>0$, there exists a choice of positive reaction rate constants 
$\bar{\kappa}$ such that the positive solutions of the system
\begin{equation}\label{eq:sistemacongamma}
 \sum_{j=1}^n \gamma_j\varphi_{\alpha,j}(\kappa)x^{a_j}-\gamma_{n+1}T_{\alpha}=0, \quad \alpha=1,\dots,m,
 \end{equation}
 are in bijection with the positive solutions of 
 \begin{equation}\label{eq:sistemasingamma}
 \ell_{\alpha}(x,\bar \kappa)-T_{\alpha}=\sum_{j=1}^n \varphi_{\alpha,j}(\bar \kappa)x^{a_j}-T_{\alpha}=0, \quad \alpha=1,\dots,m.
 \end{equation}
 Moreover, the reaction rate constants $\bar \kappa$ can be obtained from the original  constants $\kappa$ by scaling only  the rate 
 constants of those reactions coming out from a core complex. 
\end{thm}

The proofs of these two results are given below. First, we show in the network of  
Example~\ref{ej:mixedphospho} which parameters we rescale in the proof of Theorem~\ref{th:reescalamientos}.

\begin{ej}[Example~\ref{ej:mixedphospho}, continued] It is easy to check that the network of 
Example~\ref{ej:mixedphospho} with the MESSI structure defined before is an $s$-toric MESSI systems and satisfies the hypotheses of 
Theorem~\ref{th:reescalamientos}. In the proof of this theorem, we show that it is sufficient to rescale the parameters 
$k_1$, $k_4$ and $k_7$, which are the rate constants of reactions coming out from the core complexes $S_0+E$, $S_1+E$ and $S_2+F$ respectively.

\end{ej} 

We need to introduce the following sets, as in the proof of Theorem~21 of \cite{aliciaMer}. 

\begin{defi}\label{def:Lk} Let $G$ be the underlying digraph of a MESSI system. 
Consider a minimal partition of the set of species as in~\eqref{eq:partition} and the associated digraph $G_E$. 
We define the following subsets of indices:
\begin{align*}
L_0=&\{\beta \geq 1 : \text{indegree} \text{ of }\Sp^{(\beta)}\text{ is }0\},\ \text{and for } k\geq 1:\\
L_k=&\{\beta \geq 1: \text{for any edge }  \Sp^{(\gamma)}\to\Sp^{(\beta)}\text{ in }G_E \text{ it holds that }
\gamma \in L_t, \text{ with }  t<k \}\backslash \underset{t=0}{\overset{k-1}{\bigcup}} L_t.
\end{align*}
\end{defi} 

\begin{ej}[Example~\ref{ej:mixedphospho}, continued] The subsets $L_k$, $k\geq 0$ of
 Definition~\ref{def:Lk} for the network of Example~\ref{ej:mixedphospho} with the 
 MESSI structure defined before are: $L_0=\{1,2\}$ and $L_1=\{3\}$ (see the corresponding digraph $G_E$ at Figure~\ref{fig:ejdigraphsmixedphospho}).
\end{ej}

\subsection{The proofs}
We will need a series of remarks and technical lemmas in order to prove our
main result Theorem~\ref{th:reescalamientos} that ensures that the general method developed in Section~\ref{sec:2} can be applied to the
determination of  regions of multistatinarity for {\em any} $s$-toric MESSI system.  We introduce
new ideas, but unluckily our results lie heavily on the machinery developed
in~\cite{aliciaMer} and they require the reader to consult that paper. We will need some combinatorial definitions that we will recall succintly. 

We first give the proof of Proposition \ref{prop:parametrizationmessi}.

\begin{proof}[Proof of Proposition \ref{prop:parametrizationmessi}  ] 

Suppose that we have a minimal partition of the set of species as 
in~\eqref{eq:partition}.
With our assumptions, the hypotheses of Theorem~12 in \cite{aliciaMer} are satisfied, then there exist $m$ conservation laws of the form 
\begin{equation}\label{eq:consalpha}
\ell_{\alpha}(u,x)\, = \, T_{\alpha}, \text{ where } 
\ \ell_{\alpha}(u,x)= \sum_{X_j \in \Sp^{(\alpha)}} x_j + \sum_{k\in\intal(\alpha)} u_k,\quad \alpha=1,\dots,m,
\end{equation}
for some constants $T_{\alpha}$, which are positive if the trajectory intersects 
the positive orthant for each $\alpha=1,\dots,m$, and where  $\intal(\alpha)\subset\{1,\dots,p\}$ is the following set of indices:
\begin{equation*}\label{eq:intal}
 \intal(\alpha)=\{k: \exists\ y\uri U_k,\, \text{with}\ 
y\  \text{core complex with one species belonging to }\, \Sp^{(\alpha)}\}.
\end{equation*}
Because the system is an $s$-toric MESSI system, by Proposition~27 of \cite{aliciaMer}, we can obtain the concentration of the intermediate species 
at steady state in terms of the concentrations of the core species.
 That proposition states 
that there are (explicit) rational functions $\mu_k(\kappa)\in\Q(\kappa), 1\le k\le p$,  such that at steady state:
\begin{equation}\label{eq:monom}
 u_k( {\boldsymbol x}) \, = \, \mu_k(\kappa) \, {\boldsymbol x}^{y}, \quad k=1,\dots, p,
\end{equation}
where here {\boldmath{$x$}} denotes the vector of variables corresponding 
to the concentration of core species and $y$ is the unique core complex reacting through intermediates to $U_k$ (here we identify the complex $y$ with the corresponding vector in $\Z_{\geq 0}^s$).
Also, as $G_2$ is weakly reversible and $G_E$ has no direct cycles, we can apply Theorem~21 of \cite{aliciaMer} to obtain a rational parametrization of 
the concentration of the core species. 
Consider the subsets $L_k$, $k\geq 0$ as in Definition~\ref{def:Lk}. 
Observe that the set $L_0$ is not empty because $G_E$ has no direct cycles. 
Fix $x_{i_{\alpha}}\in\Sp^{(\alpha)}$ for each $\alpha=1,\dots,m$. 
In the proof of Theorem~21 of \cite{aliciaMer} it is shown that we can then parametrize all the species of $\Sp^{(\alpha)}$ for $\alpha\in L_k$ 
in terms of $x_{i_{\alpha}}$, species corresponding to core subsets in 
$L_t$ with $t<k$ and the rate constants $\tau$. Under the assumption that any pair of nodes in the same connected component of $G_2$ 
is connected by a single simple path, we can show that this parametrization is a positive monomial 
parametrization, using Theorem~28 of \cite{aliciaMer}.

Then, the concentration of a core
species at steady state can be written as a monomial in terms of the variables $x_{i_{\alpha}}$, for $\alpha=1,\dots,m$ 
and the rate constants $\kappa$, and using this and \eqref{eq:monom}, the same holds for any intermediate species. 
We denote by $\{a_1,\dots,a_n\}$ the different monomials that appear in this monomial parametrization. We replace this parametrization in the conservations 
laws \eqref{eq:consalpha} and we get a system as in \eqref{conslawparam},
where $\varphi_{\alpha,j}(\kappa)$ is the sum of the coefficients in the parametrization of the species that appear 
in the $\alpha$-th conservation law and have the monomial $x^{a_j}$, that is, $\varphi_{\alpha,j}(\kappa)$ is a positive rational function depending on the reaction rate constants $\kappa$.

\end{proof}

In order to prove Theorem \ref{th:reescalamientos} we need some lemmas. 
The following lemma shows how the values of $\tau(\kappa)$ and $\mu_k(\kappa)$ for $k=1,\dots,p$, 
depending on the reaction rate constants $\kappa$, are modified if we consider new rate constants $\bar \kappa$ 
obtained from $\kappa$ after scaling by a positive number all constants in a reaction coming out from a core complex.

\begin{lemma}\label{lemma:multiplykappa} Let $G$ be the underlying digraph of an $s$-toric MESSI system, 
with reaction rate constants $\kappa$, $\mu_k(\kappa)$ as in \eqref{eq:monom}. Fix $\ell_y\in\R_{>0}$ 
for each $y$ core complex. Consider the following reaction rate constants $\bar \kappa$ obtained from the rate constants $\kappa$: 
\begin{equation}\label{eq:barkappa}
 \bar \kappa_{yy'}=\left\lbrace \begin{array}{l c l}
                           \ell_y\kappa_{yy'} & \text{if} & y \ \text{is a core complex, } \\
                           \kappa_{yy'} & \text{if} &  y \text{ is not a core complex.}
                          \end{array}\right.
\end{equation}
That is, we multiply the reactions rate constants coming out from a core complex (we multiply by $\ell_y$ 
if the core complex is $y$) and we keep fixed the other rate constants (the constants coming 
out from an intermediate complex). Then, for each $k=1,\dots,p$ we have
\begin{equation}\label{eq:barmu}
 \mu_k(\bar \kappa)=\ell_y\mu_k(\kappa) \ \text{ if }\ y \text{ is the unique complex core such that } y\uri U_k.
\end{equation}
Consequently, if $y\xrightarrow{\tau} y'$ is in $G_1$, then 
\begin{equation}\label{eq:bartau}
\tau(\bar \kappa) = \ell_y\tau(\kappa).
\end{equation}
\end{lemma}
\begin{proof}  Following the proofs of Proposition~27 in \cite{aliciaMer} and Theorem~2 in \cite{fw13}, 
we recall how to obtain the constants 
$\mu_k(\kappa)$ for fixed reaction rate constants $\kappa$. They build a new linear labeled directed graph
$\widehat{G}$ with node set $\Sp^{(0)}\cup\{*\}$, which consists of
collapsing all core complexes into the vertex $*$,
and labeled directed edges that are obtained from hiding the
core complexes in the labels. For example, $X_i+X_j\overset{\kappa}{\rightarrow} U_k$
becomes $*\overset{\kappa x_ix_j}{\longrightarrow}U_k$ and $U_k\overset{\kappa'}{\rightarrow} X_i+X_j$
becomes $U_k\overset{\kappa'}{\longrightarrow}*$. 

They show, using the Laplacian of a graph and the Matrix-tree Theorem (see \cite{laplacian,tutte}), that 
\[\mu_k(\kappa) =\rho_k/\rho,\]
for any $k=1, \dots, p$, where 
\begin{equation*}
\rho_k=\underset{T\; an \; {{U_k}-tree}}{\sum}c^{T},\quad \rho=\underset{T\; an \; *-tree}{\sum}c^{T}.
\end{equation*}

It is easy to check that every $*$-tree involves labels in 
$\Q[\kappa]$, and only labels from edges coming out from an intermediate complex. 
As the system is $s$-toric, for every intermediate complex formed with the intermediate species $U_k$, there is a unique
core complex $y$ such that $y\uri U_k$. Then, every $U_k$-tree involves labels in terms of $\kappa$ 
and the concentrations of the species that form $y$. 
Moreover, as there must be a path from $*$ to $U_k$ in each $U_k$-tree, then, a label from an edge 
coming out from $y$ necessarily appears in each tree (and is the unique label 
from an edge coming out from a core complex). Then, if we consider the constants $\bar \kappa$, each 
label from an edge coming out from $y$ is multiplied by $\ell_y$ and then
 $\mu_k(\bar \kappa)=\ell_y\mu_k(\kappa) \ \text{ if }\ y \uri U_k$, as wanted.
 The expression of the constants $\tau (\bar \kappa)$ follows from~\eqref{eq:tau}.
\end{proof}


In the following lemma we give in more detail the form of the positive parametrization given in
 Proposition~\ref{prop:parametrizationmessi}. 

\begin{lemma}\label{lemma:concentrationcore} With the hypotheses of Theorem~\ref{th:reescalamientos}, 
fix $X_{i_{1}},\dots,X_{i_{m}}$ species as in Proposition~\ref{prop:parametrizationmessi}, 
with $X_{i_{\alpha}}\in\Sp^{(\alpha)}$, for each $\alpha=1,\dots,m$.
Take any other species $X_{i}\in\Sp^{(\alpha)}$ with $\alpha\in L_k$, $X_{i}\neq X_{i_{\alpha}}$, 
with $L_k$ as in Definition~\ref{def:Lk}. Then, the concentration of $X_{i}$ in terms of $x_{i_1},\dots,x_{i_m}$ can be expressed in the form:
\begin{equation}\label{concentrationofcore}
x_{i}=\phi(\tau)\, x_{i_{\alpha}}\, \underline{x}^{a},
\end{equation}
for some $\phi(\tau)\in\Q(\tau)$, where ${\underline{x}}^{a}$ is a monomial that depends only on variables $x_{i_{\beta}}$ with ${\beta}\in L_t$, with $t<k$.
Moreover, $\phi(\tau)$ has the form
\begin{equation}\label{concentrationofcore2}
\phi(\tau)=\left(\prod_{j=1}^q \dfrac{ \tau_{j,1}}{\tau_{j,2}}\right) g({\tau'})
\end{equation}
for some $q\geq 1$, where $g({\tau'})$ 
is a rational function of the constants ${\tau'}$, with ${\tau'}$ the label of edges of connected components of 
$G_2$ corresponding to $\Sp^{(\beta)}$, with $\beta \in L_t$, with $t<k$, 
and $\tau_{j,1}, \tau_{j,2}$ label of edges of the connected component of $G_2$ corresponding to $\Sp^{(\alpha)}$, for each $j=1,\dots,q$.
\end{lemma}
\begin{proof} We have that any pair of nodes in the component of $G_2$ corresponding to $\Sp^{(\alpha)}$ 
are connected by a unique simple path. Then, two different (simple) 
cycles can only share a node in common (if there are two nodes in common, there will be more 
than a single path connecting one of the nodes to the other, a contradiction).
For each species $X_j\in \Sp^{(\alpha)}$, we consider the set of cycles in the subgraph $G_2$ that have $X_j$ as a node, that is:
\[\mathcal{C}(X_j)=\{C: C \text{ is a (simple) cycle with } X_j \text { a node of } C  \}.\]
Observe that these sets are nonempty because $G_2$ is weakly reversible by hypothesis.
Now, we define the following subsets of ${\Sp}^{(\alpha)}$.
\begin{align*}
N_0=&\{X_{i_{\alpha}}\},\\ 
N_q=&\{X_j \in \Sp^{(\alpha)}: X_j\in C, \text{ for some }\, C\in \mathcal{C}(X_{j'}), \text{ with } 
X_{j'}\in N_{q-1}\}\backslash \underset{t=0}{\overset{q-1}{\bigcup}} N_t,\  q \geq 1.
\end{align*}
Suppose that $X_{i}\in N_q$, for some $q\geq 1$. By hypothesis, there is a unique simple path between
 two nodes species in $\Sp^{(\alpha)}$, so there exist unique 
species $Z_0=X_{i_{\alpha}}, Z_1,\dots,Z_q=X_{i}$, such that $Z_{j}\in \mathcal{C}(Z_{j-1})$, for $j=1,\dots, q$.
Then, there exist $q$ cycles in $G_2$:
\begin{figure}[h]
\centering
\begin{tikzpicture}
\node[]at(-1,1)(a){$Z_{0}$};
\node[]at(0.2,1)(b){$\cdots$};
\node[]at(1.4,1)(c){$Z_1$};
\node[]at(0.2,0)(d){$\cdots$};
\node[]at(2.6,1)(e){$\cdots$};
\node[]at(3.8,1)(f){$Z_2$};
\node[]at(2.6,0)(g){$\cdots$};
\node[]at(5,0.5)(h){$\cdots$};
\node[]at(6.2,1)(i){$Z_{q-1}$};
\node[]at(7.4,1)(j){$\cdots$};
\node[]at(7.4,0)(k){$\cdots$};
\node[]at(8.6,1)(l){$Z_q$};
\node[]at(4.6,1)(m){};
\node[]at(5.2,1)(n){};
\path[->,every node/.style={font=\sffamily\small}]
(a) edge node [above] {} (b)
(b) edge node [above] {} (c)
(c) edge [bend left] node [below] {} (d)
(d) edge[bend left] node [below] {} (a)
(c) edge node [above] {} (e)
(e) edge node [above] {} (f)
(f) edge [bend left] node [below] {} (g)
(g) edge[bend left] node [below] {} (c)
(i) edge node [above] {} (j)
(j) edge node [above] {} (l)
(f) edge node [above] {} (m)
(n) edge node [above] {} (i)
(l) edge [bend left] node [below] {} (k)
(k) edge[bend left] node [below] {} (i);
\end{tikzpicture} 
\end{figure}

\noindent each one of the form:
\begin{figure}[h]
\centering
\begin{tikzpicture}
\node[]at(-1,1)(a){$Z_{j-1}$};
\node[]at(1.5,1)(b){$\cdots$};
\node[]at(4,1)(c){$Z_j$};
\node[]at(1.5,-0.25)(d){$\cdots$};
\path[->,every node/.style={font=\sffamily\small}]
(a) edge node [above] {$\tau_{j,1}x_{h_{j,1}}$} (b)
(b) edge node [above] {} (c)
(c) edge[bend left] node [below] {\quad \qquad $\tau_{j,2}x_{h_{j,2}}$} (d)
(d) edge[bend left] node [below]  {} (a);
\end{tikzpicture} 
\end{figure}

\noindent where $x_{h_j,1}$, $x_{h_j,2}$ are the concentrations of species in core subsets belonging to $L_t$ for $t<k$ or are 
equal to $1$. 
Following the proof of Theorem~28 of \cite{aliciaMer},  we have that at steady state:
\[\tau_{j,1}x_{h_{j,1}} z_{j-1} = \tau_{j,2}x_{h_{j,2}} z_j,\]
for each $j=1,\dots,q$. From all these equations, we have that:
\begin{equation*}
x_{i}=z_q=\left(\prod_{j=1}^q \dfrac{ \tau_{j,1}}{\tau_{j,2}}\right) \left(\prod_{j=1}^q \dfrac{x_{h_{j,1}}}{x_{h_{j,2}}}\right)x_{i_{\alpha}}.
\end{equation*}
Using a recursive argument for the variables $x_{h_j,1}$, $x_{h_j,2}$ , we obtain what we wanted.
\end{proof}

In the proof of Theorem~\ref{th:reescalamientos} we will show how to modify the 
rate constants coming out from core complexes. If the digraph $G_E$ has no directed cycles,
 we can consider the sets $L_k$, $k\geq 0$ as in Definition~\ref{def:Lk}. Given $k\geq 1$ and $\alpha\in L_k$, 
 we denote by $\mathcal{Y}_{\alpha}$ the set of reactant\footnote{A reactant complex $y$ 
 is a complex for which exists a reaction $y\to y'$.} core complexes which consist only of 
 one species of $\Sp^{(\alpha)}$ or which consist of one species of $\Sp^{(\alpha)}$ and one species in a core subset
 with index in $L_t$ with $t<k$.

If $G_2$ is weakly reversible, for each $y\in \mathcal{Y}_{\alpha}$, there exist at least one 
simple cycle $C$ in $G_2$ that contains an outgoing edge of the form
$X_i\xrightarrow{\tau x_j}$
if $y=X_i+X_j$ or an edge of the form $X_i\xrightarrow{\tau}$
 if $y=X_i$, where $X_i\in\mathcal{S}^{(\alpha)}$. In this case, we say that the complex 
 $y$ \emph{appears} in the simple cycle $C$. We define the following subsets of  $\mathcal{Y}_{\alpha}$.
 
\begin{defi}\label{def:Mq} Assume $G$ is the underlying digraph of a MESSI system 
satisfying the hypotheses of Theorem~\ref{th:reescalamientos}; in particular,
we fix $X_{i_{\alpha}}\in\Sp^{(\alpha)}$ for each $\alpha=1,\dots,m$. Let $N_q$ and $\mathcal{C}(X_j)$ defined as in the proof of 
Lemma~\ref{lemma:concentrationcore}. For any $k\geq 1$ and $\alpha\in L_k$, we define the following subsets of 
$\mathcal{Y}_{\alpha}$:
\begin{align*}
M_0=&\{y\in \mathcal{Y}_{\alpha}: \text{one species of } y \text{ is } X_{i_{\alpha}}\},\\
M'_0=&\{y\in \mathcal{Y}_{\alpha}: y  \text{ appears in } C \ \text{with } C\in \mathcal{C}(X_{i_{\alpha}})\}\backslash {M}_0,,\ \text{and for } q\geq 1:\\
M_q=&\{y\in \mathcal{Y}_{\alpha}: \text{one species of } y \text{ belongs to } N_q\}\backslash \underset{t=0}{\overset{q-1}{\bigcup}} ({M}_t\cup M'_t),\\
M'_q=&\{y\in \mathcal{Y}_{\alpha}: y  \text{ appears in } C \ \text{with }
 C\in \mathcal{C}(Z), \text{for some } Z\in N_{q}\})\backslash (\underset{t=0}{\overset{q-1}{\bigcup}} ({M}_t\cup M'_t)\cup M_q). 
\end{align*}
\end{defi}

We clarify in our example the previous definitions.

\begin{ej}[Example~\ref{ej:mixedphospho}, continued] Consider the network and its 
MESSI structure of Example~\ref{ej:mixedphospho}. Choose the species $S_0\in\Sp^{(3)}$.
 Looking at the connected component corresponding to $\Sp^{(3)}$ in the digraph $G_2$ in 
 Figure~\ref{fig:ejdigraphsmixedphospho}, the sets $N_q$ that appear in the proof of 
 Lemma~\ref{lemma:concentrationcore} are: $N_0=\{S_0\}$, $N_1=\{S_1, S_2\}$. The set $\mathcal{C}(S_0)$ consists only of the simple cycle: 

\begin{center}
\begin{tikzpicture}
\node[]at(-1,0.8)(a){$S_0$};
\node[]at(0,0.8)(b){$S_1$};
\node[]at(1,0.8)(c){$S_2$};
\path[->,every node/.style={font=\sffamily\footnotesize}]
    (a) edge node [above] {$e\tau_1$} (b)
    (b) edge node [above] {$e\tau_2$} (c)
    (c) edge[bend left] node [below] {$f\tau_3$} (a);
\end{tikzpicture} 
\end{center}
The set $\mathcal{Y}_3$ is $\{S_0+E, S_1 + E, S_2+F\}$. The sets $M_q$ of
 Definition~\ref{def:Mq} are: $M_0=\{S_0+E\}$, $M'_0=\{S_1+E, S_2+F\}$ 
 (the complexes $S_1+E$ and $S_2 + F$ appear in the the previous cycle of $\mathcal{C}(S_0)$).
\end{ej}

Now we are ready to present the proof of Theorem~\ref{th:reescalamientos}.

\begin{proof}[Proof of Theorem \ref{th:reescalamientos}] We can suppose without loss of generality that the coefficient $\gamma_{n+1}$
in the system \eqref{eq:sistemacongamma} is equal to $1$; if not, we divide each equation by 
$\gamma_{n+1}$ and we obtain new values of $\gamma$ for each monomial.
Note that $x_{i_{\alpha}}$ is one of the monomials that appears in the system~\eqref{systemgamma} 
for all $\alpha=1,\dots,m$. We can  suppose that the corresponding multiplier $\gamma_{\alpha}$ of $x_{i_{\alpha}}$ in system~\eqref{systemgamma} is equal to $1$ for all $\alpha$. 
Otherwise, we change the variables
\[\gamma_{\alpha} x_{i_\alpha}=\bar x_{i_{\alpha}}.\]
In this case, we get a system with new values of the vector $\gamma$,
 in which the positive solutions are in bijection with the positive solutions of 
 system~\eqref{eq:sistemacongamma}.

With these assumptions, we assert that we can transform 
system~\eqref{eq:sistemacongamma} 
into system~\eqref{eq:sistemasingamma}, just rescaling the rate constants of reactions coming out 
from a core complex, in a certain order, multiplying each one by an appropriate constant. 
We consider the sets $L_k$, as in Definition~\ref{def:Lk}. Recall that $L_0$ is no empty because $G_E$ 
has no direct cycles. Because the partition is minimal the subsets of core species $\Sp^{(\alpha)}$ are in bijection 
with the connected components of $G_2$ and the set of nodes of the corresponding component equals $\Sp^{(\alpha)}$. 


%


Let $\Sp^{(\alpha)}\in L_0$.  We showed in Lemma~\ref{lemma:concentrationcore} that all the 
core species in $\Sp^{(\alpha)}$ can be written in terms of the monomial $x_{i_{\alpha}}$, reaction
 rate constants and no other variables. If an intermediate complex has its unique core complex reacting to it
via intermediates formed with species only in $\Sp^{(\alpha)}$, 
then the concentration of the corresponding intermediate species also depends only on $x_{i_{\alpha}}$ 
and reaction rate constants. That is, all the concentrations of these species have $x_{i_{\alpha}}$ as the corresponding monomial in the parametrization. 
We supposed that in system~\eqref{eq:sistemacongamma} the monomial $x_{i_{\alpha}}$ is 
multiplied by $\gamma_{\alpha}=1$, then, there is nothing to rescale.
%

Now we proceed recursively. Fix $k\geq 1$. Suppose that we have already rescaled properly the reaction rate constants of edges coming out from core complexes whose parametrizations 
depends only on variables $x_{i_{\beta}}$ with $\beta\in L_t$, with $t<k$. Fix one core subset $\Sp^{(\alpha)}$, with $\alpha\in L_k$. 
We will show how to rescale the rate constants of reactions coming out from complexes in the set $\mathcal{Y}_{\alpha}$, defined above.

The digraph $G_2$ is weakly reversible, then we can consider the sets $M_q, M'_q$, $q\geq 0$, as in Definition~\ref{def:Mq}.
We are going to rescale the rate constants of reactions coming out from a complex in $M_0$, then in $M'_0$, then in $M_1$ and so on, in that order.
First, we show how to modify the constants of reactions coming out from a complex in $M_0$. 
Because the system is $s$-toric, each intermediate complex has a unique core complex reacting through
 intermediates to it. We consider the intermediates complexes such the unique core complex reacting through 
 intermediates to it is in $M_0$ or in $M'_0$ (if there is no one, we don't rescale anything). Suppose then 
 that there is one intermediate complex formed by an intermediate species $U_{\ell}$ such that $y \uri U_{\ell}$, with $y\in M_0$ or $y\in M'_0$.
If the core complex $y\in M_0$, then $y=X_{i_{\alpha}}$ or $y=X_{i_{\alpha}}+X_j$,  
with $X_j$ in a core subset belonging to $L_t$, with $t<k$. If $y=X_i$, the concentration of $U_k$ is 
$u_{\ell}=\mu_{\ell}(\kappa)x_{i_{\alpha}}$, with $\mu_{\ell}$ as in \eqref{eq:mu},
and we are assuming that the monomial $x_{i_{\alpha}}$ is  multiplied by $\gamma_{\alpha}=1$. If $y=X_{i_{\alpha}}+X_j$, then we can write:
\[u_{\ell}=\mu_{\ell}(\kappa)x_{i_{\alpha}}x_j.\]
Now, $x_j$ is a concentration of a core species and its parametrization can be written
in terms of cores species of 
 subsets in the partition with indices in $L_t$,
with $t<k$ and reaction rate constants $\tau'(\kappa)$, with $\tau'(\kappa)$
 labels of edges of connected components of $G_2$, corresponding to core subsets with indices in $L_t$,
  with $t<k$. We write $x_j=g({\tau'(\kappa)})\underline{x}^{a}$, with $\underline{x}^{a}$ a
   monomial in these other species and $g$ a rational function, and we get:
\[u_{\ell}=\mu_{\ell}(\kappa) g({\tau'(\kappa)}) \underline{x}^{a} x_{i_{\alpha}}.\]
Suppose that the monomial $\underline{x}^{a}x_{i_{\alpha}}$ appears in system~\eqref{eq:sistemacongamma} 
multiplied by $\gamma$. Then, we want new reaction rate constants $\bar \kappa$ such that:
\begin{equation}\label{eq:mu}
\gamma \mu_{\ell}(\kappa) g({\tau'(\kappa)})= \mu_{\ell}(\bar \kappa) g({{\tau'(\bar \kappa)}}).
\end{equation}
To ease the notation, we denote $\mu_{\ell}=\mu_{\ell}(\kappa)$, $\tau'=\tau'(\kappa)$, 
$\bar \mu_{\ell}=\mu_{\ell}(\bar \kappa)$, $\bar{{\tau}}'=\tau'(\bar \kappa)$ (and we will denote with a bar the constants 
depending on $\bar \kappa$ and without a bar, the
 constants depending on $\kappa$). The constants ${\tau}'$ have been modified previously by hypothesis (note that a constant $\tau$ can only appear in one edge of $G_2$, 
because of the condition that $G_E$ has no cycles) and replaced by the constants $\bar{{\tau}}'$. It is clear that we can do the rescaling: it is enough to
multiply each reaction constant of reactions coming out from the core complex $y$ by the constant:
$\gamma\frac{g({\tau'})}{g(\bar{{\tau}}')}.$
Then, by Lemma~\ref{lemma:multiplykappa}, we obtain the equality \eqref{eq:mu}. 
Now, if $y\in M'_0$, $y$ appears in $C$ with $C\in \mathcal{C}(X_{i_{\alpha}})$. Then, $y=X_i$ or $y=X_i+X_{j'}$, with $C$ of the form
\begin{figure}[h]
\centering
\begin{tikzpicture}
\node[]at(-1,1)(a){$X_{i_{\alpha}}$};
\node[]at(1,1)(b){$\cdots$};
\node[]at(3,1)(c){$X_i$};
\node[]at(1,0)(d){$\cdots$};
\path[->,every node/.style={font=\sffamily\small}]
(a) edge node [above] {$\tau_{1}x_j$} (b)
(b) edge node [above] {} (c)
(c) edge[bend left] node [below] {\quad \qquad $\tau_{2}x_{j'}$} (d)
(d) edge[bend left] node [below]  {} (a);
\end{tikzpicture} 
\end{figure}

\noindent where $x_{j'}$ is the concentration of $X_{j'}$ or is equal to $1$ (if $y=X_i$), and similarly for $x_j$. Then, we have at steady state:
\[u_{\ell}=\mu_{\ell}\, x_i\, x_{j'}=\mu_{\ell}\frac{\tau_{1}}{\tau_{2}}\, x_{i_{\alpha}}\, x_j\]
That is, $u_{\ell}$ depends on the concentrations of the species of the complex $X_{i_{\alpha}} + X_j$, 
which belongs to $M_0$. We then modify the reaction rate constants coming out of $X_{i_{\alpha}} + X_j$ 
multiplying it  by an appropiate constant in a similar way as we did in the previous case, looking at the value of $\gamma$ 
that appears in the corresponding monomial (note that if we modified these constants before, the previous 
rescaling also works for this case). Note that when later we modify the constants of the complex $y=X_i+X_{j'}$ 
which belongs to $M'_0$ (we will see how to do this), the rescaling will be coherent. That is, if we multiply the 
constants of each reaction coming out from $X_{i_{\alpha}} + X_j$ by $\nu_1$, and the constants coming out 
from $y$ by $\nu_2$ the rescaling will be coherent if we have:
\[{\bar \mu_{\ell}}\, \frac{\bar \tau_{1}}{\bar \tau_{2}}=\nu_1\,  { \mu_{\ell}}\, \frac{\tau_{1}}{ \tau_{2}},\]
but this holds by Lemma~\ref{lemma:multiplykappa}:
\[{\bar \mu_{\ell}}\, \frac{\bar \tau_{1}}{\bar \tau_{2}}={\nu_2 \, \mu_{\ell}}\,\frac{\nu_1 \, \tau_{1}}{\nu_2\, \tau_{2}}=\nu_1\,  { \mu_{\ell}}\, \frac{\tau_{1}}{ \tau_{2}},\]
where $\bar \mu_{\ell}$, $\bar \tau_{1}$, $\bar \tau_{2}$ denotes the values of the functions $\mu_{\ell}$, 
$\tau_{1}$, $\tau_{2}$ corresponding to the new constants $\bar \kappa$.
We modify all the reactions rate constants coming out of complexes $y$ belonging to $M_0$ in this way:
looking at intermediates complexes $U_{\ell}$ such that $y\uri  U_{\ell}$ or $y'\uri  U_{\ell}$, with $y'\in M'_0$ 
and such that in the parametrization of the intermediate species appears the monomial corresponding to the complex $y$. 
If there is no such intermediate complex we multiply the constants by $1$.
Also, we observe that with this rescaling, we modified all the constants $\tau$ that label an edge in $G_2$ 
of the form $X_{i_{\alpha}}\xrightarrow{\tau x_j}$. 

Now, we show how to rescale the constants of complexes in $M'_0$. Let $y\in M'_0$, then $y=X_i$ or 
$y=X_i+X_{j'}$ and we have a cycle as we showed previously in this proof. By Lemma~\ref{lemma:concentrationcore}, the concentration $x_i$ is of the form
\[x_i=\frac{\tau_{1}}{\tau_{2}}g({\tau'})\underline{x}^{a}x_{i_{\alpha}},\]
where $\tau'$ are labels of edges of 
connected components of $G_2$, corresponding to core subsets belong to $L_t$, with $t<k$, $g({\tau})$ a 
rational form and $\underline{x}^{a}$ a monomial in variables in core subsets belong to $L_t$, with $t<k$.
The constant $\tau_1$ and the constants ${\tau'}$ have been already modified by the constants ${\bar{\tau_1}}$ 
and ${\bar{\tau}'}$ respectively.  
It is clear that we can do the rescaling if we modify $\tau_2$. If $\gamma$ is the constant that multiplies the
monomial $\underline{x}^{a}x_{i_{\alpha}}$ in 
system~\eqref{eq:sistemacongamma}, we want 
\[\gamma\frac{\tau_{1}}{\tau_{2}}g({\tau'})=\frac{\bar{\tau_{1}}}{{\bar \tau_{2}}}g({\bar{\tau}'}),\]
and we get this equality if we multiply each rate constant of  a reaction coming out from $y$ by the constant
$\frac{\bar{\tau_1}g(\bar{{\tau}}')}{\gamma\tau_1g({\tau'})}$ and we apply Lemma~\ref{lemma:multiplykappa}.
We do this for all complexes in $M'_0$.

We proceed recursively rescaling the remaining constants of reactions coming out from complexes in $M_q$, 
and then from $M'_q$, for each $q$. We first modify the constants of complexes in $M_q$, by looking at the 
parametrization of intermediate species as we did when we showed how to rescale the constants of reactions 
coming out from complexes in $M_0$. After that, we modify the constants of complexes in $M'_q$ by looking 
the concentration of the core species in $\Sp^{(\alpha)}$ that appear in the complex, as we did for the complexes in $M'_0$. 
Then, we can rescale all the complexes of $\mathcal{Y}_{\alpha}$, for all $\alpha\in L_k$. We can proceed for all $k$, $k\geq 1$, in order, and then we are done.

\end{proof}

\begin{ej}  
	The distributive multisite phosphorylation systems showed in Section~\ref{sec:3} are all in the hypotheses of Theorem~\ref{th:reescalamientos}. 
	A MESSI structure of the network for the double phosporylation ($n=2$) is given by this minimal partition of the species:

$\Sp^{(0)}=\{ES_0, ES_1, FS_1, FS_2\}$  (the intermediate species), $\Sp^{(1)}=\{E\}$, $\Sp^{(2)}=\{F\}$ and $\Sp^{(3)}=\{S_0, S_1, S_2\}$. 

	The digraphs $G_1$, $G_2$ and $G_E$ are depicted in Figure \ref{fig:digraphs2phospho}. 
	It is easy to check the conditions of Theorem~\ref{th:reescalamientos} in this case. 
	Following the proof of this theorem, we can show which parameters are sufficient to rescale. 
	For this case is sufficient to modify $k_{\rm{on}_0}$, $k_{\rm{on}_1}$,
	$\ell_{\rm{on}_0}$ and $\ell_{\rm{on}_1}$, the rate constants of reactions coming out of core complexes.
	
	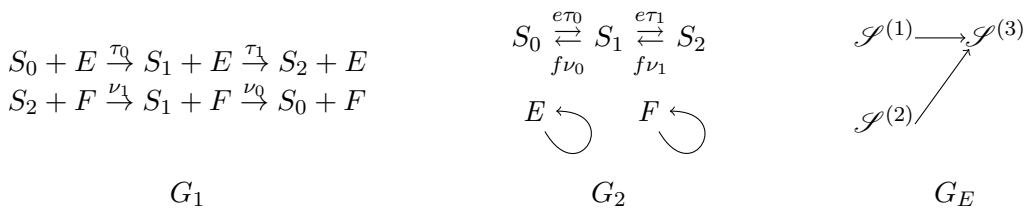
\begin{figure}
		\centering
		\begin{tikzpicture}
		\node[]at(0,0)(a){$\begin{array}{l}
			S_0 + E  \overset{\tau_0}{\rightarrow} S_1 + E \overset{\tau_1}{\rightarrow}
			S_2 + E\\
			S_2 + F  \overset{\nu_1}{\rightarrow} S_1 + F \overset{\nu_0}{\rightarrow}
			S_0 + F
			\end{array}$};
		\node[]at(0,-1.5){$G_1$};
		\end{tikzpicture} 
		\hspace{1.0cm}
		\begin{tikzpicture}
		\node[]at(0,0.5)(a){$\begin{array}{l}
			S_0 \underset{f\nu_0}{\overset{e\tau_0}{\rightleftarrows}} S_1 \underset{f\nu_1}{\overset{e\tau_1}{\rightleftarrows}} S_2
			\end{array}$};
		\node[]at(-1,-0.4)(b){E};
		\node[]at(0.5,-0.4)(c){F};
		\path
		(b)   edge[in=0,out=300,loop] node  {} (b);
		\path
		(c)   edge[in=0,out=300,loop]node  {} (c);
		\node[]at(0,-1.5){$G_2$};
		\end{tikzpicture} 
		\hspace{1.0cm}
		\begin{tikzpicture}[ampersand replacement=\&] 
		\matrix (m) [matrix of math nodes, row sep=1.5em, column sep=1em, text height=1.5ex, text depth=0.25ex]
		{ \Sp^{(1)}\&  \Sp^{(3)}  \\
			\Sp^{(2)} \&  \& \\};
		\draw[->]($(m-1-1)+(0.4,0)$) to node[below] (x) {} ($(m-1-2)+(-0.4,0)$);
		\draw[->]($(m-2-1)+(0.4,0)$) to node[below] (x) {} ($(m-1-2)+(-0.35,-0.15)$);
		\node[]at(0,-1.5){$G_E$};
		\end{tikzpicture}
		\caption{The digraphs $G_1$, $G_2$ and $G_E$ of the double phosphorylation.}
		\label{fig:digraphs2phospho}
	\end{figure}

\end{ej}

  \section*{Acknowledgment} The authors are grateful to the Kurt and Alice Wallenberg Foundation and to the Institut 
Mittag-Leffler, Sweden, for their
 support to start and to make progress on this work. Our thanks go also to the Mathematics Department of the Royal Institute of 
Technology, Sweden,  for the wonderful hospitality we enjoyed. AD and MG are partially supported by UBACYT 
20020100100242, CONICET PIP 11220150100473, and ANPCyT PICT 2013-1110, Argentina.

    \end{document}